\documentclass[copyright,creativecommons]{eptcs}
\providecommand{\event}{ACT 2019}
\usepackage{breakurl}
\usepackage{underscore}

\usepackage{amsmath}
\usepackage{amsfonts}
\usepackage{amsthm}
\usepackage{float}
\usepackage{mathtools}
\usepackage{tikz}
\usepackage{subcaption}
\usepackage{enumerate}
\usetikzlibrary{arrows,automata}
\usetikzlibrary{positioning,calc,decorations.pathmorphing,decorations.pathreplacing,patterns}
\newcommand{\cC}[0]{%
\mathcal{C}}
\newcommand{\cD}[0]{%
  \mathcal{D}}

\DeclareMathOperator{\Nodes}{Nodes}
\DeclareMathOperator{\Edges}{Edges}
\DeclareMathOperator{\Ob}{Ob}
\DeclareMathOperator{\Mor}{Mor}

\newcommand{\Mon}[0]{\mathbf{Mon}}
\newcommand{\Nom}[0]{\mathbf{Nom}}
\newcommand{\Cat}[0]{\mathbf{Cat}}

\makeatletter
\pgfdeclareshape{triangle}{
    \inheritsavedanchors[from=rectangle]
    \inheritanchorborder[from=rectangle]
    \inheritanchor[from=rectangle]{center}
    \inheritanchor[from=rectangle]{north}
    \inheritanchor[from=rectangle]{south}
    \inheritanchor[from=rectangle]{east}
    \inheritanchor[from=rectangle]{southwest}
    \inheritanchor[from=rectangle]{southeast}

    \backgroundpath{
        \southwest \pgf@xb=\pgf@x \pgf@yb=\pgf@y
        \northeast \pgf@xa=\pgf@x \pgf@ya=\pgf@y
        \pgf@xc=0.5\pgf@xa \advance\pgf@xc by+0.5\pgf@xb

        \pgfpathmoveto{\pgfpoint{\pgf@xc}{\pgf@ya}}
        \pgfpathlineto{\pgfpoint{\pgf@xb}{\pgf@yb}}
        \pgfpathlineto{\pgfpoint{\pgf@xa}{\pgf@yb}}
        \pgfpathlineto{\pgfpoint{\pgf@xc}{\pgf@ya}}
    }
}
\makeatother

\tikzstyle{big-triangle}=[triangle,draw,inner sep=0pt,minimum width=0.5cm,minimum height=0.15cm]
\tikzstyle{medium-triangle}=[big-triangle,scale=0.7]
\tikzstyle{vbig-triangle}=[big-triangle,scale=2]
\tikzstyle{mbig-triangle}=[big-triangle,scale=1.2]
\tikzstyle{small-triangle}=[big-triangle,scale=0.25]

\title{Autonomization of monoidal categories}
\author{Antonin Delpeuch
\institute{Department of Computer Science\\
University of Oxford}
\email{antonin.delpeuch@cs.ox.ac.uk}}

\def\titlerunning{Autonomization of monoidal categories}
\def\authorrunning{A. Delpeuch}

\begin{document}

\maketitle

\newtheorem{defi}{Definition}
\newtheorem{lemma}{Lemma}
\newtheorem{prop}{Proposition}
\newtheorem{thm}{Theorem}
\newtheorem{coro}{Corollary}
\newtheorem{conj}{Conjecture}
\renewcommand{\vec}[1]{%
  \overrightarrow{#1}}

\begin{abstract}
We show that contrary to common belief in the DisCoCat community, a monoidal
category is all that is needed to define a categorical compositional model of natural language. This relies on a construction which freely adds adjoints to a monoidal category. In the case of distributional semantics, this broadens the range of available models, to include non-linear maps and cartesian products for instance. We illustrate the applications of this principle to various distributional models of meaning.
\end{abstract}

\section{Introduction}

The DisCoCat
model~\cite{clark2008compositional,coecke2010mathematical} applies
category theory to linguistics. The idea is to view the type-logical
representation of grammatical structure~\cite{lambek2008pregroup} as
a morphism in a compact closed category, and use this morphism to
combine individual word meanings into sentence representations.  This
gives an inherently compositional template for models of meaning which
respect the grammatical structure by design. The flagship application
of this idea is its interpretation it in vector-based categories, as
it makes it applicable to distributional semantics (hence called the
\textbf{Dis}tributional \textbf{Co}mpositional \textbf{Cat}egorical
model). Distributional semantics represent word meanings by computing
co-occurence statistics of words across large
corpora~\cite{schutze1998automatic}, or more recently by training
neural networks to learn word embeddings~\cite{mikolov2013efficient}.
The task of generalizing these representations from words to larger
syntactical units such as sentences is a challenging task, and the
promise of the DisCoCat approach is to tackle it with a categorical
viewpoint.

The mantra of DisCoCat is that a model of meaning is given by a
compact closed category and a choice of word meanings in
it~\cite{coecke2010mathematical}. This definition is flexible in that
it encompasses both logical models of meaning and distributional ones,
depending on the concrete category
chosen~\cite{clark2008compositional}. However, the requirement to
define a compact closed category is still a restricting one, as it
rules out a wide range of categories where objects do not always have
adjoints. The general consensus in the community was that when it
comes to vector-based models of meanings, DisCoCat requires the
semantics to be assembled out of linear maps
only~\cite{wijnholds2018classical,lewis2019compositionality}. We show
how to lift this restriction, by only requiring the semantic category
to be monoidal, and freely adding the caps and cups required for the
interpretation of the grammatical
structure~\cite{delpeuch2014autonomization}. We briefly present these
results in Section~\ref{sec:autonomization}.

The bulk of this piece consists in showing how various models of
meaning formulated outside the DisCoCat framework can actually be
recast in it, thanks to the free addition of adjoints. Our aim is to
understand why the DisCoCat community has pushed itself in the narrow
corner of linear models for such a long time, in a context where the
wider linguistics community has overwhelmingly embraced non-linear
models.  By conducting this analysis, our hope is to encourage the
wider applied category theory community to confront their models with
the actual practitioners in the field targeted, and challenge any
categorical dogma with their feedback.

\section{Free autonomous categories} \label{sec:autonomization}

We recall a few useful definitions. The full details of the equations
satisfied by the following structures, as well as a presentation
of their graphical languages can be found in \cite{selinger2010survey}.

\begin{defi}
  A \textbf{monoidal category} is a category $\cC$
  equipped with a symmetric bifunctor $ \_ \otimes \_ : \cC \times \cC
  \rightarrow \cC$. This operation is furthermore required to have
  a unit $I \in \cC$ and to be naturally associative.
\end{defi}

\noindent A monoidal category is symmetric when its monoidal product is.

\begin{defi}
  An \textbf{autonomous} (or rigid) \textbf{category} is a monoidal
  category such that any object $A$ has left and right adjoints, meaning
  that there are morphisms $\epsilon_l : A^l \otimes A \rightarrow 1$,
  $\epsilon_r : A \otimes A^l \rightarrow 1$, $\eta_l : 1 \rightarrow A \otimes A^l$
  and $\eta_r : 1 \rightarrow A^r \otimes A$ satisfying equations.
\end{defi}

\noindent In a symmetric monoidal category, left and right adjoints are isomorphic. An autonomous
category that is also symmetric is called a \textbf{compact closed category}.
An example of such category is $(\mathbf{Vect}, \otimes, I)$, the category
of finite-dimensional vector spaces and linear maps, with the tensor product
as monoidal structure.

\begin{defi}
  A \textbf{cartesian category} is a category with all finite products.
  In other words it is a symmetric monoidal category with a natural family of
  copying and discarding maps satisfying the usual laws.
\end{defi}

\noindent An example of cartesian category is $(\mathbf{Vect}, \oplus,
O)$, the category of finite-dimensional vector spaces and linear maps,
with the direct sum as monoidal structure. Another canonical example
is $(\mathbf{Set}, \times, \{ * \})$, the category of sets and
functions equipped with the (set-theoretic) cartesian product.

\begin{prop} 
  If a cartesian category is also autonomous for the same monoidal structure,
  then every object is isomorphic to the monoidal unit.
\end{prop}

\noindent In other words, cartesian and autonomous structures are incompatible.

\begin{thm} \label{thm:autonomization}
  Any monoidal category $(\mathcal{C}, \_ \otimes \_, I)$ gives rise
  to a free autonomous category $(L(\mathcal{C}), (\_, \_), ())$. The
  embedding functor $F : \mathcal{C} \rightarrow L(\cC)$ is
  strong monoidal and faithful.
\end{thm}

This means that it is possible to freely add caps and cups to any
monoidal category. To emphasize the different nature of the monoidal
structures on both categories, we use different notations for
them. They are related by the embedding functor $F$ which is strong
monoidal. Informally, the category $L(\mathcal{C})$ is defined by
taking formal string diagrams annotated with morphisms of
$\mathcal{C}$. This result seems to known as folklore in some
communities. For the sake of rigour we provide a detailed construction
in Appendix~\ref{sec:construction}.

\begin{figure}
  \centering
  \begin{tikzpicture}[scale=1.3,
every node/.style={%
text height=1.5ex,text depth=0.25ex,}]
\makeatletter

\pgfdeclareshape{triangle}{
    \inheritsavedanchors[from=rectangle]
    \inheritanchorborder[from=rectangle]
    \inheritanchor[from=rectangle]{center}
    \inheritanchor[from=rectangle]{north}
    \inheritanchor[from=rectangle]{south}
    \inheritanchor[from=rectangle]{east}
    \inheritanchor[from=rectangle]{southwest}
    \inheritanchor[from=rectangle]{southeast}

    \backgroundpath{
        \southwest \pgf@xb=\pgf@x \pgf@yb=\pgf@y
        \northeast \pgf@xa=\pgf@x \pgf@ya=\pgf@y
        \pgf@xc=0.5\pgf@xa \advance\pgf@xc by+0.5\pgf@xb

        \pgfpathmoveto{\pgfpoint{\pgf@xc}{\pgf@ya}}
        \pgfpathlineto{\pgfpoint{\pgf@xb}{\pgf@yb}}
        \pgfpathlineto{\pgfpoint{\pgf@xa}{\pgf@yb}}
        \pgfpathlineto{\pgfpoint{\pgf@xc}{\pgf@ya}}
    }
}

\tikzstyle{big-triangle}=[triangle,draw,inner sep=0pt,minimum width=0.5cm,minimum height=0.15cm]
\tikzstyle{medium-triangle}=[big-triangle,scale=0.7]
\tikzstyle{vbig-triangle}=[big-triangle,scale=2]
\tikzstyle{mbig-triangle}=[big-triangle,scale=1.2]
\tikzstyle{small-triangle}=[big-triangle,scale=0.25]

\node at (0.0,0) (t2) {$n$};
\node at (0.0,1.5) (w0) {Pat};
\node at (1.5,0) (t8) {$n^r$ $s$ $n^l$};
\node at (1.5,1.5) (w1) {grows};
\node at (3.2,0) (t17) {$n$ $n^l$};
\node at (3.0,1.5) (w2) {delicious};
\node at (4.5,0) (t28) {$n$};
\node at (4.5,1.5) (w3) {kiwis};



\draw[->] (0.0,-0.25) .. controls (0.0,-0.55525) and (0.24474999999999997,-0.8) .. (0.55,-0.8) .. controls (0.8552500000000001,-0.8) and (1.1,-0.55525) .. (1.1,-0.25);

\draw[<-] (1.9,-0.25) .. controls (1.9,-0.5552499999999999) and (2.14475,-0.7999999999999998) .. (2.45,-0.7999999999999998) .. controls (2.75525,-0.7999999999999998) and (3,-0.5552499999999999) .. (3,-0.25);

\draw[->] (1.5,-0.25) -- (1.5,-1.1);

\draw[->] (4.5,-0.25) arc (0:-180:0.55cm);

\node at (1.5,-1.3) {$s$};


\node[medium-triangle] at (0.0,0.75) (fresh) {};
\draw (0.0,0.25) -- (fresh);

\node[rectangle,draw,scale=0.8,minimum width=0.4cm] at (1.5,0.62) (fresh) {$f_1$};

\coordinate (rep) at (1.1,0.25);
\draw (rep) -- (rep |- fresh.north);
\draw[->] (rep |- fresh.north) arc (180:0:0.15cm);
\coordinate (rep) at (1.9,0.25);
\draw (rep) -- (rep |- fresh.north);
\draw[->] (rep |- fresh.north) arc (0:180:0.15cm);

\draw (1.5,0.25) -- (fresh);

\node[medium-triangle] at (4.5,0.75) (fresh) {};
\draw (4.5,0.25) -- (fresh);

\begin{scope}[xshift=-3.9cm]
\node[rectangle,draw,minimum width=0.65cm,scale=0.8] at (6.9,0.6) (n) {$+$};
\draw[->] (n) -- ($(n.center)-(0,0.35)$);
\draw[<-] ($(n.north)+(0.1,0)$) arc (180:0:0.15cm);
\draw ($(n.north)+(0.4,0)$) -- (7.3,0.25);
\node[medium-triangle] at ($(n.north) +(-0.1,0.4)$) (tri) {};
\draw[<-] ($(n.north)-(0.1,0)$) -- (tri);
\end{scope}

\begin{scope}[xshift=7cm,yshift=-1cm]
    \node at (-1.3,1) {{\Large $=$}};

\node at (0.0,0) {$s$};

\node[rectangle,draw,scale=0.8,minimum width=1cm] at (0,0.7) (fresh) {$f_1$};
\draw[<-] (0,0.25) -- (fresh);

\node[medium-triangle] at (0.1,1.8) (td1) {};
\node[medium-triangle] at (0.4,1.8) (td2) {};

\node[rectangle,draw,scale=0.8,minimum width=0.7cm,fill=white] at (0.25,1.2) (n1) {$+$};

\draw[->] (n1) -- (n1 |- fresh.north);
\draw[->] (td1) -- (td1 |- n1.north);
\draw[->] (td2) -- (td2 |- n1.north);

\node[medium-triangle] at (-0.25,1.2) (tg) {};
\draw[->] (tg) -- (tg |- fresh.north);

\end{scope}

\end{tikzpicture}
  \caption{Eliminating free caps and cups from a sentence meaning}
  \label{fig:octo}
\end{figure}

Morphisms in $L( \cC)$ are of little use for semantics as they are
formal objects. Luckily, in the particular case where the domain and
codomain are objects of $\cC$, we can eliminate the formal caps and
cups introduced to recover a morphism in $\cC$.  This is an equivalent
of the normalization property of~\cite{preller2007free} for our free
construction.

\begin{thm} \label{thm:fullness}
  The functor $F$ is full.
\end{thm}

\begin{proof}
  Let $f \in L \cC(F(A),F(B))$ with $A,B \in \cC$, we show that $f \in F(\cC(A,B))$.
  If no formal caps and cups occur in $f$, then it can be expressed
  as a vertical and horizontal composition of generator morphisms from
  $\cC$, so it belongs to $F(\cC(A,B))$.
  Therefore we only need to show that all caps and cups can be eliminated
  from a diagram whose boundaries do not contain any adjoints.
  As no generator contains adjoints in their domain or codomain,
  the winding number of any wire in the diagram $f$ is null.
  For any unit occuring in $f$ we can find a matching counit
  on the same wire, such that the unit and counit can be cancelled together.
  This elimination property is shown as Lemma~3.12 in \cite{mimram20143dimensional}.
\end{proof}

 What fullness of $F$ means is that any
morphism in $L \cC$ between objects coming from $\mathcal{C}$
actually comes from $\mathcal{C}$ too. Because sentence meanings in
DisCoCat are interpreted as elements of the sentence space $S$, they
are therefore morphisms in $L \cC ([I],[S])$. By fullness this
means that they correspond to morphisms $\mathcal{C}(I,S)$ (and a
unique one by faithfulness). This means that sentence meanings are not
formal objects: they belong to the original semantic monoidal
category.

The use of $L \mathcal{C}$ as a model of meaning is illustrated in
Figure~\ref{fig:octo}. The recipe is simple:
\begin{itemize}
\item define word meanings, using formal caps to produce elements
  of the types dictated by the grammar;
\item compose word meanings with the formal cups determined
  by the type reduction witnessing the grammaticality of the sentence;
\item eliminate the formal caps and cups in the resulting diagram. Theorem~\ref{thm:fullness} guarantees that all formal caps and cups can be eliminated.
\item you obtain an element of the sentence space in the original category $\mathcal{C}$, which
  is the representation of the sentence.
\end{itemize}
In the sequel, we illustrate the concrete use of this result, showing
how it can expand the range of models of meaning which can be
formulated in the DisCoCat framework.  We also argue that it weakens
the case for $\mathbf{Vect}_\otimes$, which has so far been considered
as the canonical (if not only) semantic category suited for
distributional compositional semantics.

\section{Some early evidence against tensors}

Let us first recall how the DisCoCat model and its tensor-based
interpretation came into existence.  The interest in the tensor
product for distributional semantics was initially suggested
by~\cite{clark2007combining}, inspired by ideas from cognitive
science.  The idea was that tensors can account for correlations
between the meanings of some words. This was demonstrated by a thought
experiment called the \emph{pet fish problem}
\cite{aerts2005theory}. The idea was that the conjunction of two
concepts is not always best represented by the independent
superposition of both concepts: there is some interaction between the
two concepts. It was reported that subjects found the image of a
\emph{guppy} was neither a prototypical image of \emph{pet} nor that
of a \emph{fish}. Inspired by quantum mechanics, it was proposed to
represent concepts in Hilbert space, such that state entanglement
could be used to model this phenomenon.

This quantum-inspired model of meaning became all the more motivated
when the connection to compact closed categories was made
by~\cite{clark2008compositional,coecke2010mathematical}. Indeed,
vector spaces form a compact closed category if the tensor product is
used as monoidal structure, but this fails for other natural
structures such as the direct sum, as explained in the previous
section.  All this evidence, coming simultaneously and independently
from cognitive science and category theory, seemed to indicate in an
unambiguous way that $\mathbf{Vect}_\otimes$ was the natural
category to develop distributional compositional semantics.

A few difficulties with this proposal became soon clear, and the first
one was the high dimensionality of the word representations
involved. Indeed, the number of parameters of word representations
grows exponentially with the length of its grammatical type. This
means for instance that ditransitive verbs such as \emph{give}, which
are applied to a subject and two objects, would be represented by $n^3
s$ parameters, where $n$ is the dimensionality of the sentence space
and $n$ that of the noun space. Paradoxically, this prohibitive cost
was not seen as a fundamental problem of the model - it somehow
justified the absence of large-scale benchmarks against competing
approaches. It also suggested that the model was more expressive or
powerful, having more parameters than other models: this expressivity
was only waiting for the right hardware (such as quantum devices) to
be unleashed~\cite{zeng2016quantum}.

However, various attempts were made to learn the high-order word
representations suggested by the
framework~\cite{grefenstette2013multistep,kartsaklis2014study,grefenstette2015concrete}.
Still, these experiments were only conducted for simple syntactic
structures such as \emph{adjective-noun} or \emph{subject-verb-object}
patterns, far from the coverage of arbitrary text generally expected
in the field natural language processing.

One of these approaches was that of~\cite{polajnar2014reducing}, where
it was proposed to restrict verb meanings to particular shapes, reducing
the number of parameters to learn. The intention was to stick to the DisCoCat
framework using $\mathbf{Vect}_\otimes$, but impose further restrictions
on the semantic spaces and word meanings to obtain more tractable and effective
models. Four approaches to model transitive verbs in \emph{subject-verb-object}
sentences were compared:
\begin{enumerate}[(i)]
\item \textbf{Tensor}: the standard, unrestricted representation of
  transitive verbs as tensors of type $N \otimes N \otimes S$, where
  $N$ and $S$ are sentence and noun spaces. The space $S$ is
  two-dimensional, the two truth values corresponding to orthogonal
  basis vectors.
\item \textbf{KKMat}: the same approach, but collapsing $S$ to a
  single dimension, the truth value being represented by the norm of
  the vector.
\item \textbf{SKMat}: verbs are represented matrices in $N \otimes S$
  (where $S$ is two dimensional again). They are multiplied by the
  object vector seen as a diagonal matrix, and then contracted with
  the subject vector.
\item \textbf{2Mat}: verbs are represented by two $N \otimes S$
  tensors. They are multiplied with the subject and object
  respectively to obtain two vectors in $S$, which are finally
  concatenated into a four-dimensional vector.
\end{enumerate}

\begin{figure}[H]
  \centering
  \begin{subfigure}{0.45\textwidth}
    \centering
    \begin{tikzpicture}[scale=2]
      \node[vbig-triangle,minimum width=1cm,minimum height=0.5cm] (t) at (0,0) {};
      \node (rep) at (-.3,-.5) {$N$};
      \draw (rep) -- (rep |- t.south);
      \node (rep) at (.3,-.5) {$N$};
      \draw (rep) -- (rep |- t.south);
      \node (rep) at (0,-.5) {$S$};
      \draw (rep) -- (rep |- t.south);
    \end{tikzpicture}
    \caption{\textbf{Tensor}, the most generic form for a transitive verb}
  \end{subfigure}
  \begin{subfigure}{0.45\textwidth}
    \centering
    \begin{tikzpicture}[scale=2]
      \node[vbig-triangle,minimum width=1cm,minimum height=0.5cm] (t) at (0,0) {};
      \node (rep) at (-.3,-.5) {$N$};
      \draw (rep) -- (rep |- t.south);
      \node (rep) at (.3,-.5) {$N$};
      \draw (rep) -- (rep |- t.south);
      \node[gray] (rep) at (0,-.5) {$(S)$};
    \end{tikzpicture}
    \caption{\textbf{KKMat}, assuming that $S = I$}
  \end{subfigure}
  \begin{subfigure}{0.45\textwidth}
    \centering
    \begin{tikzpicture}[scale=2]
      \node[rectangle,draw,minimum width=0.5cm,minimum height=0.5cm] (t) at (0,-.1cm) {};
      \node[circle,draw,inner sep=2pt] (c) at (-.15,.4cm) {};
      \draw (t.north) -- +(0,.1cm) .. controls ($(c)+(.15,-.2)$) and ($(c)+(.15,-.1)$) .. (c);
      \node (rep) at (-.3,-.5) {$N$};
      \node (j) at ($(t.north)+(0,.1)$) {};
      \draw (rep) -- (rep |- j) .. controls ($(c)+(-.15,-.2)$) and ($(c)+(-.15,-.1)$) .. (c);
      \node (rep) at (.3,-.5) {$N$};
     \draw (rep) -- (rep |- c.north) edge[bend right=90] (c);
      \node (rep) at (0,-.5) {$S$};
      \draw (rep) -- (rep |- t.south);
      \node[node distance=.55cm,above of=t,fill=white,inner sep=1pt] {$N$};
    \end{tikzpicture}
    \caption{\textbf{SKMat}, where the circular node is the Frobenius copy map for the canonical basis}
  \end{subfigure}
    \begin{subfigure}{0.45\textwidth}
    \centering
    \begin{tikzpicture}[scale=2]
      \node[rectangle,draw,red,minimum width=0.7cm,minimum height=0.5cm] (t) at (0,-.15cm) {$c$};
      \node[rectangle,draw,minimum width=0.3cm,minimum height=0.5cm] (a) at (-.12,.45cm) {};
      \node[rectangle,draw,minimum width=0.3cm,minimum height=0.5cm] (b) at (.12,.45cm) {};
      \node (rep) at (-.3,-.5) {$N$};
      \node (j) at ($(t.north)+(0,.1)$) {};
      \draw (rep) -- (rep |- a.north) edge[bend left=90] (a.north);
      \node (rep) at (.3,-.5) {$N$};
      \draw (rep) -- (rep |- b.north) edge[bend right=90] (b.north);
      \node (rep) at (0,-.5) {$2S$};
      \draw (rep) -- (rep |- t.south);
      \draw (a.south) -- (a.south |- t.north);
      \draw (b) -- (b |- t.north);
      \node[below of=a,inner sep=2pt,fill=white,node distance=0.6cm] {$S$};
            \node[below of=b,inner sep=2pt,fill=white,node distance=0.6cm] {$S$};
    \end{tikzpicture}
    \caption{\textbf{2Mat}, where $c : S \otimes S \rightarrow 2S$ is the concatenation function (does not exist)}
  \end{subfigure}
    \caption{The four approaches to dimensionality reduction studied by~\cite{polajnar2014reducing}}
    \label{fig:2mat}
\end{figure}
These four approaches, represented as string diagrams in Figure~\ref{fig:2mat}, were evaluated on a task which consists in
estimating the plausibility of a candidate sentence and the best model
overall was \textbf{2Mat}. What the authors did not realize however
was that unlike the three other approaches, \textbf{2Mat} cannot be
formulated in $\mathbf{Vect}_\otimes$ since it involves concatenating
vectors, which is not a bilinear operation. In
$\mathbf{Vect}_\otimes$, any morphism taking two arguments as inputs
is necessarily bilinear in them, since it is a linear map from the
tensor product of their spaces. So there is no function $c : S \otimes S \rightarrow S \oplus S$ such that $c(a \otimes b) = (a,b)$ for all $a,b \in S$.  The fact that the best performing
approach to dimensionality reduction in DisCoCat was a model that
actually violated the assumptions of DisCoCat could have been seen as
a sign that the assumptions of the model should have been revisited, but
it actually went largely unnoticed in the community.

The reason why the \textbf{2Mat} model was allowed to slip out of
$\mathbf{Vect}_\otimes$ is that it was formulated in terms of function
application rather than tensor contraction. These two views are
equivalent in a compact closed category such as
$\mathbf{Vect}_\otimes$ but only when the appropriate monoidal product $\otimes$
is used to combine the arguments together.  We argue that thinking in
terms of function application to design models of meaning is much more
intuitive than using tensor contraction, especially to linguists not
trained in quantum physics.

All the models proposed by~\cite{polajnar2014reducing}, including
\textbf{2Mat}, can be recast as DisCoCat models in
$L \mathbf{Set}_\times$, the free autonomous category on
$\mathbf{Set}_\times$, the category of sets and functions with the
cartesian product. In this setting, there are no tensors to contract:
meaning composition can only be thought of in terms of function
application. In fact, the models of meaning which can be formulated in
$L \mathbf{Set}_\times$ are precisely all the functional models directed
by the grammar.

\section{Further evidence against linearity}

Beyond the requirement to use tensors to pair up the arguments of any
morphism with multiple inputs, $\mathbf{Vect}_\otimes$ also requires that
morphisms to be linear. This requirement has not been challenged much
in the community either. This is perhaps because linearity assumptions
are pervasive in many fields of science; non-linear models can often
be reduced to linear ones via some transformation. In the particular
context of $\mathbf{Vect}_\otimes$ however, there is little room for such
adaptations. In this section, we show why this might be a problem in
its own right.

In linear algebra, an affine map is a map $f : x \mapsto Ax + b$,
where $A$ is a linear map and $b$ is a constant vector. Affine maps
are so similar to linear maps that the two notions are often
conflated: $Ax + b$ is coloquially called a linear combination of the
elements of $x$. In fact, any affine map $f : X \rightarrow Y$ can be
seen as a linear map $f' : X + \mathbf{1} \rightarrow Y$, where
$\mathbf{1}$ is the one-dimensional vector space for the same field:
$f(x) = f'(x, 1)$. With this observation, it is possible to adapt the
matrix calculus of linear maps to work for affine maps.

However, this trick can sadly not be used to generalize
$\mathbf{Vect}_\otimes$ to affine maps. The product $(A + \mathbf{1}) \otimes
(B + \mathbf{1})$ is not equal to $(A \otimes B) + \mathbf{1}$ for
non-null spaces and there does not seem to be a natural generalization
of bilinearity to allow for constants. Affine maps and tensor products
just do not mix up well. This shows how stringent the requirement to
use $\mathbf{Vect}_\otimes$ is: even the seemingly innocuous addition of a
constant breaks the machinery down.

This vow to strict linearity is of course problematic in a context
where neural networks have become the machine learning models of
choice in natural language processing, as these models crucially rely
on nonlinearities to function. We will review the possible
interactions between DisCoCat and neural networks in the next section,
but let us first give linearity the benefit of the doubt: it could be
that non-linearities are needed to learn word vectors, but the
resulting vectors could then be composed linearily when deriving the
representation of larger text
units~\cite{milajevs2014evaluating,sadrzadeh2017quantization}.

One of the seminal results that popularized neural word embeddings in
linguistics was that of \cite{mikolov2013efficient}.  After learning
word embeddings from text with a simple model that is trained to
guess a missing word in a fixed-length textual context, they observe
that the learned vectors satisfy some promising semantic
properties. The classical example of such a property is the fact that
$\vec{King} - \vec{Man} + \vec{Woman}$ gives a vector whose nearest
neighbour is $\vec{Queen}$. This sort of relation between vectors,
which arises from the training process without being enforced directly
by the learning objective, invites to model predicates as additive
functions.

This insight proved successful in a related problem, that of learning
embeddings for entities in a knowledge graph. A knowledge graph is
essentially a directed multigraph where edges are annotated with
properties.  These edges are seen as \emph{triples} formed by their
source, predicate and target, each encoding a fact about the world,
such as the triple (\texttt{Cantaloupe\_Island\_(song)},
\texttt{composer}, \texttt{Herbie\_Hancock}) for instance.  Graph
embeddings associate to each entity in the knowledge graph a vector in
a finite-dimensional vector space, computed solely from the knowledge
graph itself. Such embeddings are then useful to tackle various
information retrieval tasks, such as completing the knowledge graph by
adding missing triples. One of the popular approaches to compute such
embeddings is the \textbf{TransE} model~\cite{bordes2013translating},
which consists in optimizing the learning objective that $v(o)- v(s)
\simeq v(p)$ for each triple $(s,p,o)$ (subject - predicate - object).

Whether word vectors are learned from text or from a knowledge graph,
it is therefore tempting to model predicates in a distributional
compositional model by such additive functions. For instance, the
representation of the \emph{royal} adjective could be a function
adding $\vec{Queen} - \vec{Woman}$ to its argument. Sadly, this is again
not a linear function but an affine one: it cannot be implemented
in $\mathbf{Vec}_\otimes$. Again, it is a perfectly valid semantic
representation in $L \mathbf{Set}_\times$.

\section{Recasting convolutional models in DisCoCat}

Neural networks are popular to learn word vectors, but the dominant
architectures to learn these word vectors ignore grammar entirely: the
networks are simply trained to predict words in a text from their
surrounding context, regardless of their grammatical
function~\cite{mikolov2013efficient}. Some of these models, such as
the ones designed to translate text between languages~\cite{sutskever2014sequence},
make use of vector representations for larger units of texts such as
sentences: however, these vectors are rarely designed to encode the
full meaning of the textual unit, as the translation to the target
language relies on \emph{attention mechanisms} which rely on the
individual vectors as well as the combined representation to produce
the desired output~\cite{bahdanau2014neural}. As such, they do not fully address
the problem of combining distributional representations.

However, neural networks have also been applied in syntax-aware
approaches which construct sentence representations from word
representations. In~\cite{socher2013recursive}, a small neural network
is applied recursively along the syntax tree of the sentence to derive
a vector representation of each node from that of its
children. \cite{lewis2019compositionality} proposed to recast this
approach in the DisCoCat framework, which would give a similar model
driven from a pregroup-induced representation of the grammar
instead. But doing so required them to exchange the neural network for
a tensor in order to formulate the model in
$\mathbf{Vect}_\otimes$. Choosing to formulate it in
$L \mathbf{Set}_\times$ instead lets us keep the original, non-linear
convolutional unit. In effect, this just amounts to swapping the context-free
grammar used by~\cite{socher2013recursive} for a type-driven grammar such
as a pregroup grammar. In fact, a similar proposal was formulated
by \cite{krishnamurthy2013vector} earlier, using Combinatory Categorial Grammar
instead of pregroup grammar. This proposal also falls into the scope of DisCoCat,
when implemented in $L \mathbf{Set}_\times$.

\section{Conclusion}

We have shown that contrary to common belief in the DisCoCat community,
a monoidal category is all that is needed to define a compositional model
of meaning. In the case of distributional semantics, this broadens the range
of available models to include non-linear maps and cartesian products.
We hope this will encourage the community to experiment with alternate models,
which should be more amenable to competing with the state of the art in natural
language processing.

\section{Acknowledgements}

The author wishes to thank Amar Hadzihasanovic, Andrew Pitts, Anne
Preller, Bob Coecke, Dan Marsden, David Reutter, Jamie Vicary, Konstantinos
Meichanetzidis, Paul-André Melliès and Samuel Mimram for their useful feedback on
this work, part of which was conducted while at École normale
supérieure, Paris. The author is supported by an EPSRC Studentship.

\bibliographystyle{eptcs}
\bibliography{zotero}

\appendix








\section{Constructing autonomous categories} \label{sec:construction}

\subsection{Idea of the construction}

Let $\mathcal{C}$ be a strict monoidal category.  Our goal is to embed
$\mathcal{C}$ in a ``larger'' category, $L(\cC)$, which will be
autonomous. The embedding has to be functorial, so that the original
composition operations are retained.

The $\epsilon$ and $\eta$ maps will be purely formal,
which means that they will have no interpretation in the original
category.  Our approach to define them consists in taking diagrams
seriously: the arrows of our autonomous category will be diagrams. To do
so, we adapt the definitions of~\cite{joyal1988planar}, who defined the
diagrams for autonomous categories and proved their soundness and
completeness.  They assume for simplicity that the links in the diagrams
are piecewise linear and we follow their choice.

Our construction is similar to that of~\cite{preller2007free} who
defined the autonomous category freely generated by a category $\cC$,
not assumed to be monoidal.  Taking into account the monoidal
structure of $\cC$ construction is important for the linguistic
applications which motivate our work, as we need some
compatibility between the initial monoidal structure in $\cC$ and the
monoidal structure of the larger category $L(\cC)$.
 
\subsection{Graphs}

In this section, we summarize the definitions
from~\cite{joyal1988planar} needed to define our category $L(\cC)$. Our
goal is not to give the most general definition of topological graphs,
but only to define precisely the objects that we will manipulate.

A \textbf{graph} $\Gamma$ is a compact subset of $\mathbb{R}^2$ together
with a subset $\Gamma_0 \subseteq \Gamma$ such that
\begin{enumerate}[(i)]
  \item $\Gamma_0$ is discrete and finite, its elements are called
    \textbf{nodes};
  \item $\Gamma - \Gamma_0$ has a finite number of connected components
    called \textbf{1-cells} and each of them is homeomorphic to an open
    interval.
\end{enumerate}
       
We denote by $[p; q]$ the segment connecting two points $p, q \in
\mathbb{R}^2$. For all tuples of points $t = (p_0, \dots, p_n)$ we
define $[p_0, \dots, p_n] = [p_0,p_1] \cup [p_1,p_2] \cup \dots \cup
[p_{n-1},p_n]$.  The latter is called a \textbf{piecewise linear
  segment}.  The tuple $t$ is called \textbf{reduced} when there is no
$0 < i < n$ such that $p_i \in [p_{i-1},p_{i+1}]$.  In this case, the
$p_i$ for $0 < i < n$ are called \textbf{singular points}.  A
parametrization of a reduced piecewise linear segment $\gamma : [0;1]
\to [p_0, \dots, p_n]$ such that $\gamma(0) = p_0$ and $\gamma(1) =
p_n$ is called a \textbf{piecewise linear curve}.  The
\textbf{initial} (respectively \textbf{terminal}) segment of such a
$\gamma$ is $[p_0, p_1]$ (respectively $[p_{n-1},p_n]$). We denote by
$\bar{\gamma}$ the reversed parametrization: $\bar{\gamma}(t) =
\gamma(1-t)$.

A \textbf{piecewise linear graph} $\Gamma$ is a graph where the closure
of any 1-cell is a piecewise linear segment, and such that no initial or
terminal segment is horizontal. The \textbf{edges} of such a graph are
the parametrizations of these closures, identified up to monotonous
reparametrization (hence a 1-cell gives rise to two edges, $\gamma$ and
$\bar{\gamma}$).  The set of edges is denoted by
$\Edges(\Gamma)$. In the rest of this paper, all graphs are assumed to
be piecewise linear.

Let $\gamma$ be an edge and $x < y \in [0;1]$ be preimages of
consecutive singular points of $\gamma([0;1])$. The segment
$[\gamma(x),\gamma(y)]$ is \textbf{directed top} (respectively
\textbf{bottom}) when the second coordinate of $\gamma(y)$ is greater
(respectively smaller) than that of $\gamma(x)$.  The last requirement
of the definition of a piecewise linear graph implies that initial and
terminal segments of edges are either directed top or bottom. This
allows us to define the \textbf{inputs} of a node $x$ as the set of
edges $\gamma$ such that $\gamma(1) = x$ and the terminal segment of
$\gamma$ is directed bottom. Similarly, the \textbf{outputs} of $x$ are
the edges $\gamma$ such that $\gamma(0) = x$ and the initial segment of
$\gamma$ is directed bottom.

A graph $\Gamma$ is \textbf{between slices} $a$ and $b$, where $a < b$ are reals,
when $\Gamma \subset \mathbb{R} \times [a; b]$, and such that every node
in $\mathbb{R} \times \{a\}$ (respectively $b$) has one input and no
output (respectively one output and no input). These nodes included in
$\mathbb{R} \times \{a,b\}$ are called \textbf{outer nodes} and the
others are \textbf{inner nodes}. The set of inner nodes of a graph $\Gamma$ is
denoted by $\Nodes(\Gamma)$. The reason for this notation is that
the outer nodes will not represent morphisms but simply ``gates'', i.e.
inputs and outputs of the diagram. The outer nodes included in
$\mathbb{R} \times \{a\}$ are called \textbf{lower outer nodes} and the
other outer nodes are called \textbf{upper outer nodes}.  A
\textbf{regular slice} is a $c \in [a, b]$ such that $\Gamma_0 \cap
\mathbb{R} \times \{ c \} = \emptyset$. A \textbf{unit graph} is a
graph included in $[0,1]^2$ and between slices $0$ and $1$.

Finally, we need to define the turning number $\rho(\gamma)$ of an edge
$\gamma$.  Informally, this is the number of half-turns of the edge in
the direct orientation, minus the number of half-turns in the indirect
orientation. We invite the interested reader to
consult~\cite{joyal1988planar} for a rigorous definition. The following
examples should be enough to grasp the idea:
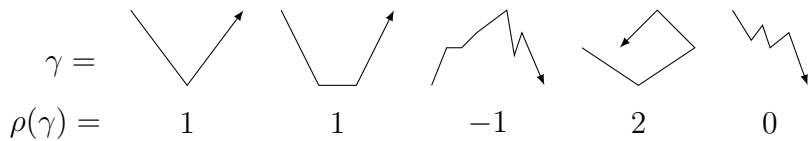
\begin{figure}[H]
  \centering
  \begin{tikzpicture}
    \draw[-latex] (0.5,0) -- (1.25,-1) -- (2,0);

    \draw[-latex] (2.5,0) -- (3,-1) -- (3.5,-1) -- (4,0);
    
    \draw[-latex] (4.5,-1) -- (4.7,-0.5) -- (4.9,-0.5) -- (5.1,-0.3) -- (5.5,0) -- (5.6,-0.6) -- (5.7,-0.3) -- (6,-1);

    \draw[-latex] (6.5,-0.5) -- (7.25,-1) -- (8,-0.5) -- (7.5,0) -- (7,-0.5);

    \draw[-latex] (8.5,0) -- (8.75,-0.4) -- (8.9,-0.2) -- (9,-0.5) -- (9.25,-0.3) -- (9.5,-1);

    \node at (-0.3,-0.75) {$\gamma = $};
    \node at (-0.5,-1.5) {$\rho(\gamma) =$};
    \node at (1.25,-1.5) {$1$};
    \node at (3.25,-1.5) {$1$};
    \node at (5.25,-1.5) {$-1$};
    \node at (7.25,-1.5) {$2$};
    \node at (9,-1.5) {$0$};
    \end{tikzpicture}
\end{figure}

\begin{defi} \label{def:yankable}
   A \textbf{yankable graph} is a graph between slices $a$ and $b$ such
   that for every edge $\gamma$ between two inner nodes, $\rho(\gamma) =
   0$.
\end{defi}

\noindent The reason for this additional requirement $\rho(\gamma) = 0$
is that we will attribute a morphism of $\cC$ to each node in
Section~\ref{sec:valued}.  Informally, as the domain and the codomain of
such morphisms cannot contain \emph{adjoints}, it is necessary that the
links between them can be \emph{yanked} to a straight line.

Here are a few examples of yankable and not yankable graphs. The edges
that make the graphs not yankable are drawn with dotted lines.
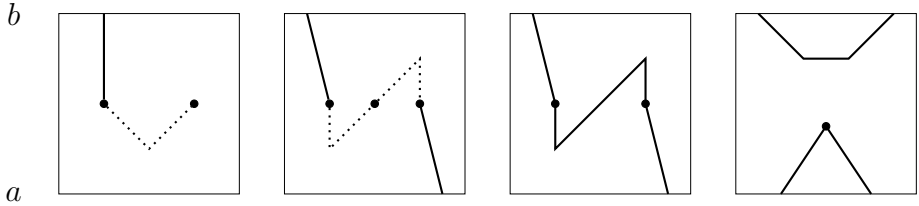
\begin{figure}[H]
   \centering
   \begin{tikzpicture}[every node/.style={node distance=0.3cm},scale=0.6]
     \node at (-1,0) {$a$};
     \node at (-1,4) {$b$};

      \draw[very thin] (0,0) rectangle (4,4);

      \tikzstyle{point}=[circle,draw,fill,inner sep=1pt];
      \node at (1,4) (gate1) {};
      \node at (3,4) (gate2) {};
      \node[point] at (1,2) (f) {};
      \node[point] at (3,2) (g) {};
      \node at (2,1) (id) {};
      \node at (2,0) (gate3) {};

      \draw[thick] (gate1.center) -- (f);
      \draw[thick,dotted] (f) -- (id.center) -- (g);

      \begin{scope}[xshift=5cm]
      \draw[very thin] (0,0) rectangle (4,4);

      \tikzstyle{point}=[circle,draw,fill,inner sep=1pt];
      \node[point] at (1,2) (f) {};
      \node[point] at (3,2) (g) {};
      \node[point] at (2,2) (h) {};
      \draw[thick] (0.5,4) -- (f);
      \draw[thick,dotted] (f) -- (1,1) -- (h) -- (3,3) -- (g);
      \draw[thick] (g.center) -- (3.5,0);

      \end{scope}
      \begin{scope}[xshift=10cm]
      \draw[very thin] (0,0) rectangle (4,4);

      \tikzstyle{point}=[circle,draw,fill,inner sep=1pt];
      \node[point] at (1,2) (f) {};
      \node[point] at (3,2) (g) {};
      \draw[thick] (0.5,4) -- (f.center) -- (1,1) -- (3,3) -- (g.center) -- (3.5,0);

      \end{scope}
      \begin{scope}[xshift=15cm]
      \draw[very thin] (0,0) rectangle (4,4);

      \tikzstyle{point}=[circle,draw,fill,inner sep=1pt];

      \node[point] at (2,1.5) (f) {};
      \draw[thick] (1,0) -- (f.center) -- (3,0);
      \draw[thick] (0.5,4) -- (1.5,3) -- (2.5,3) -- (3.5,4);

      \end{scope}
   \end{tikzpicture}
\end{figure}

A \textbf{deformation of graphs} $\Gamma$ to $\Gamma'$ is a regular
deformation of polarised graphs between $\Gamma$ and $\Gamma'$, as
defined in~\cite{joyal1988planar}. As it preserves the turning number
of edges, it preserves yankable graphs.

\subsection{Occurrences and replacement} \label{sec:replacement}

We define what an occurrence of a graph $G_1$ in a graph $\Gamma$ is,
and what the substitution of $G^1$ by $G^2$ in $\Gamma$ is. This will be
useful to define an equivalence relation on graphs, which will be
required to define the autonomous category $L(\cC)$ properly.

This notion is not needed to obtain the soundness and correctness
results of~\cite{joyal1988planar}, so one could wonder why we introduce
it while dealing with the same objects. The reason is that the
autonomous category we are constructing cannot be completely free, as we
have to retain the equalities holding in the original monoidal
category. This will enable us to define a functorial embedding in
Section~\ref{sec:func-embed}.

Let $t = (a,b,c,d)$ with $a < b$ and $c < d$ be reals. We define an
homeomorphism $\phi_t : [0, 1]^2 \to [a, b] \times [c, d]$ by
$\phi_t(x,y) = (a + x(b-a),c + y(d-c))$.  Let $G^1$ and $G^2$ be unit
graphs, such that the outer nodes of $G^1$ and $G^2$ are the same.  An
\textbf{occurrence} of $G^1$ in a graph $\Gamma \subset \mathbb{R}^2$
is a quadruplet of reals $t = (a,b,c,d)$, $a < b$ and $c < d$, such
that $\phi_t(G^1) = \Gamma \cap ([a,b] \times [c,d])$, and such that
no node of $\Gamma$ is included in the boundary of $[a,b] \times
[c,d]$.  We define $\Gamma [ G^1 \coloneqq G^2 ]_t = (\Gamma -
\phi_t(G^1)) \cup \phi_t(G^2)$, with nodes $(\Gamma_0 - \phi_t(G^1_0))
\cup \phi_t(G^2_0)$.

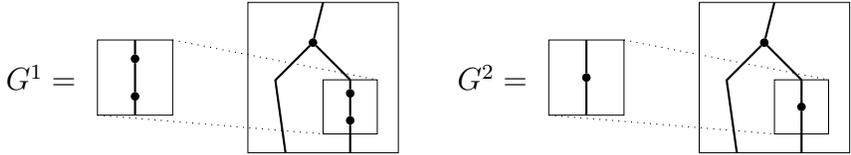
\begin{figure}[H]
  \centering
  \begin{tikzpicture}
     \tikzstyle{point}=[circle,draw,fill,inner sep=1pt];

     \foreach \xs/\ys/\sca/\lbl in {0cm/0.5cm/0.25/g1-1,3cm/0.25cm/0.18/g1-2} {
      \begin{scope}[xshift=\xs,yshift=\ys,scale=\sca]
      \draw[very thin] (0,0) rectangle (4,4);
      \node at (0,0) (corner1-\lbl) {};
      \node at (4,4) (corner2-\lbl) {};
      \node at (2,0) (gate3-\lbl) {};
      \node at (2,4) (gate1-\lbl) {};
      \node[point] at (2,3) (f) {};
      \node[point] at (2,1) (g) {};
  
      \draw[thick] (gate1-\lbl.center) -- (f) -- (g) -- (gate3-\lbl.center);

      \end{scope}
      }
      \node at (-0.75,1) {$G^1 =$};

    \foreach \xs/\ys/\sca/\lbl in {6cm/0.5cm/0.25/g2-1,9cm/0.25cm/0.18/g2-2} {
      \begin{scope}[xshift=\xs,yshift=\ys,scale=\sca]
      \draw[very thin] (0,0) rectangle (4,4);
      \node at (0,0) (corner1-\lbl) {};
      \node at (4,4) (corner2-\lbl) {};
      \node at (2,0) (gate3-\lbl) {};
      \node at (2,4) (gate1-\lbl) {};

      \node at (2,4) (gate1) {};
      \node[point] at (2,2) (f) {};
      \node at (2,0) (gate3) {};

      \draw[thick] (gate1.center) -- (f) -- (gate3.center);

      \end{scope}
     }
       \node at (5.25,1) {$G^2 =$};

      \foreach \xs/\ys/\lbl in {2cm/0cm/g1,8cm/0cm/g2} {
        \begin{scope}[xshift=\xs,yshift=\ys]
      \draw[very thin] (0,0) rectangle (2,2);

      \draw[dotted] (corner1-\lbl-1.center) -- (corner1-\lbl-2.center);
      \draw[dotted] (corner2-\lbl-1.center) -- (corner2-\lbl-2.center);

      \node at (0,0) (south) {};
      \draw[thick] (gate3-\lbl-2.center) -- (gate3-\lbl-2.center |- south);
      \node[point,above left of=gate1-\lbl-2,node distance=0.7cm] (p1) {};
      \draw[thick] (gate1-\lbl-2.center) -- (p1) -- ++(-0.5,-0.5) -- (0.5,0);
      \draw[thick] (p1) -- (1,2);
      \end{scope}
      }

   \end{tikzpicture}
\end{figure}

One can check that $\Gamma[G_1 \coloneqq G_2]_t$ is a piecewise linear
graph.  However, it is not yankable in general. We will get this
guarantee with the notion of valued graphs introduced in the next section.

\subsection{Valued graphs} \label{sec:valued}

In this section, we add \emph{valuations} to the objects introduced in
the previous section. This consists in labelling the nodes and the edges
with objects and arrows from a category, in a consistent way.

\begin{defi} \label{def:valued}
  A \textbf{$\cC$-valued graph} is a yankable graph $\Gamma$ with
  functions
  \begin{align*}
    v_0 : \Nodes(\Gamma) \to \Mor(\cC) & & v_1 : \Edges(\Gamma) \to
    \Ob(\cC) \times \mathbb{Z}
  \end{align*}
  such that:
  \begin{enumerate}[(i)]
    \item $\forall \gamma \in \Edges(\Gamma), v_1(\bar{\gamma}) =
      (A,n+\rho(\gamma))$ where $(A,n) = v_1(\gamma)$
    \item $\forall x \in \Nodes(\Gamma), \forall \gamma \in
      \Edges(\Gamma)$
        \begin{itemize}
            \item if $\gamma(0) = x$ then $v_1(\gamma) = (A,0)$ for some
              $A \in \Ob(\cC)$
            \item if $\gamma(1) = x$ then $v_1(\bar{\gamma}) = (A,0)$
              for some $A \in \Ob(\cC)$
        \end{itemize}
            \item $\forall x \in \Nodes(\Gamma), v_0(x) :
              v_1(\bar{\gamma_1}) \otimes \dots \otimes
              v_1(\bar{\gamma_p}) \to v_1(\delta_1) \otimes \dots
              \otimes v_1(\delta_q)$

                  where $\gamma_1, \dots, \gamma_p$ and $\delta_1,
                  \dots, \delta_q$ are the ordered lists of the input
                  and output edges of $x$, and where $v_1^e(\gamma) =
                  (A,0)$ is identified with $A$ for simplicity.
  \end{enumerate}
\end{defi}

Informally, the value $v_1(\gamma)$ represents the domain of the edge
$\gamma$, and $v_1(\bar{\gamma})$ represents its codomain. The condition
(i) states the relation between the two.  Note that a valued graph is
always yankable, as the rotation number of an edge between two inner
nodes is $0$ because of the requirements (i) and (ii).

The valuation of an upper outer node $x$ is $v_1(\gamma)$ where
$\gamma$ is the only edge such that $\gamma(0) = x$.  Similarly, the
valuation of a lower outer node $x$ is $v_1(\bar{\gamma})$ where
$\gamma$ is the only edge such that $\gamma(1) = x$. In both cases, we
denote this valuation by $v_b(x)$.  The \textbf{domain} of a valued
graph is the tuple $(v_b(x_1), \dots, v_b(x_p))$ where
$x_i$ is the $i$-th upper outer node of the graph.
The \textbf{codomain} is defined similarly with the lower outer nodes.

We can define the replacement of the valued graphs $G^1$ by
$G^2$ in the valued graph $\Gamma$, when the domains and codomains
of $G^1$ and $G^2$ are the same and the valuations of the nodes
of $G^1$ and $\Gamma$ agree.

\begin{lemma}
  Let $\Gamma$ be a valued graph and $G^1$, $G^2$ be valued unit graphs.
  Suppose that $t$ is an
  occurrence of $G^1$ in $\Gamma$, that the valuations of the nodes of $\Gamma$ and
  $\phi_t(G^1)$ are identical and that the domains and codomains
  of $G^1$ and $G^2$ match. Then $\Gamma[G^1 \coloneqq G^2]_t$
  can be given a valuation, such that it agrees with the valuation of
  $\Gamma$ on the nodes and edges included in $\Gamma - \phi_t(G^1)$, and it
  agrees with the valuation of $\phi_t(G^2)$ on the inner nodes and edges
  included in $\phi_t(G^2)$.
\end{lemma}

\begin{proof}
  We admit that if $\gamma : [0,1] \to \Gamma$ is an edge and $t \in
  [0,1]$ is such that $\gamma(t)$ is not a singular point of
  $\gamma([0,1])$, then $\rho(\gamma) = \rho(\gamma_{| [0,t] }) +
  \rho(\gamma_{| [t,1]})$.

  Up to a regular transformation, we can assume that the outer nodes of
  $\phi_t(G^1)$ are not singular nodes in $\Gamma$. Then, for every edge $\gamma$ in
  $\Gamma$, there is a series of $t_0 < \dots < t_k$ such that $t_0 =
  0$, $t_k = 1$, and for $0 < i < k$, $\gamma(t_k)$ is an outer node of
  $\phi_t(G^1)$, and such that $\gamma((t_i, t_{i+1}))$ is either
  included in $\phi_t(G^1)$ or in $\Gamma - \phi_t(G^1)$.

  
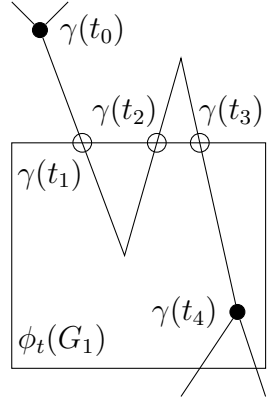
\begin{figure}[H]
  \centering
  \begin{tikzpicture}[every node/.style={node distance=0.7cm},scale=0.75]
    \draw (-0.5,0) rectangle (4,4);
    \node[circle,draw,fill,inner sep=2pt] at (0,6) (a) {};
    \node[circle,draw,fill,inner sep=2pt] at (3.5,1) (b) {};
    \draw (0.5,6.5) -- (a) -- (-0.5,6.5);
    \draw (2.5,-0.5) -- (b) -- (4,-0.5);

    \node at (1.5,2) (p1) {};
    \node at (2.5,5.5) (p2) {};
    \draw (a) -- (p1.center) -- (p2.center) -- (b);
    
    \coordinate (s1) at (0,4);
    \coordinate (s2) at (4,4);
    
    \node[circle,draw,inner sep=2.5pt] at (intersection of s1--s2 and a--p1) (t1) {};
    \node[circle,draw,inner sep=2.5pt] at (intersection of s1--s2 and p1--p2) (t2) {};
    \node[circle,draw,inner sep=2.5pt] at (intersection of s1--s2 and p2--b) (t3) {};

    \node[right of=a] {$\gamma(t_0)$};
    \node[left of=b] {$\gamma(t_4)$};
    \begin{scope}[every node/.style={node distance=0.6cm},xscale=4]
    \node[below left of=t1] {$\gamma(t_1)$};
    \node[above left of=t2] {$\gamma(t_2)$};
    \node[above right of=t3] {$\gamma(t_3)$};
    \end{scope}

    \node at (0.4,0.5) {$\phi_t(G_1)$};
    
  \end{tikzpicture}
\end{figure}
  \noindent Moreover, identifying
  temporarily $v_1$ and $v_b$ with their second projection, we have:
  \begin{align*}
  & \rho(\gamma_{|[t_0,t_1]}) = v_b(\gamma(t_1)) - v_1(\gamma)
    \\ \text{for }0 < i < k-1 \text{,\,\,}& \rho(\gamma_{|[t_i,
        t_{i+1}]}) = v_b(\gamma(t_{i+1})) - v_b(\gamma(t_i))\\ &
    \rho(\gamma_{|[t_{k-1},t_k]}) = v_1(\bar{\gamma}) -
    v_b(\gamma(t_{k-1}))
  \end{align*}

  Now for any edge $\gamma$ in $\Gamma[G^1 \coloneqq G^2]_t$, we have
  such a decomposition and $\rho(\gamma) = \sum_{i = 0}^{k-1}
  \rho(\gamma_{|[t_i, t_{i+1}]})$.  If $\gamma((0,t_1)) \subset \Gamma -
  \phi_t(G^2)$, we give $\gamma$ the valuation of the edge in $\Gamma$ starting
  on $\gamma(0)$ and whose image includes $\gamma((0,t_1))$.  If
  $\gamma((0,t_1)) \subset \phi_t(G^2)$, then similarly we give $\gamma$ the
  valuation from $G^2$. We can decompose
  the rotation number $\rho(\gamma)$ as follows:
  \begin{align*}
    \rho(\gamma) &= \sum_{i = 0}^{k-1} \rho(\gamma_{|[t_i, t_{i+1}]})
    \\ &= v_b(\gamma(t_1)) - v_1(\gamma) + \sum_{i = 1}^{k-2}
    \big(v_b(\gamma(t_{i+1})) - v_b(\gamma(t_i))\big) +
    v_1(\bar{\gamma}) - v_b(\gamma(t_{k-1})) \\ &= v_1(\bar{\gamma}) -
    v_1(\gamma)
  \end{align*}
  Hence the condition (i) of the definition is satisfied. The two other
  conditions are directly inherited from $\Gamma$ and $G^2$.
\end{proof}

\subsection{The category of valued graphs} \label{sec:category}

Our autonomous category $L(\cC)$ will be defined as the category of
$\cC$-valued graphs. But to do so, we need to define an equivalence
relation to account for some equalities of arrows from $\cC$.  This
consists in defining reduction rules based on replacement.

We will need to define some particular $\cC$-valued graphs.  Instead of
defining the graph and the valuations separately, we choose to draw
them, replacing the black points representing our nodes by boxes
containing the valuation of these nodes. The valuations of the edges are
dropped when they are clear from the context.

For all $f \in \cC(A_1 \otimes \dots \otimes A_n,B_1 \otimes \dots
\otimes B_p)$ and $g \in \cC(B_1 \otimes \dots \otimes B_p,C_1 \otimes
\dots \otimes C_q)$ we define the following graphs:
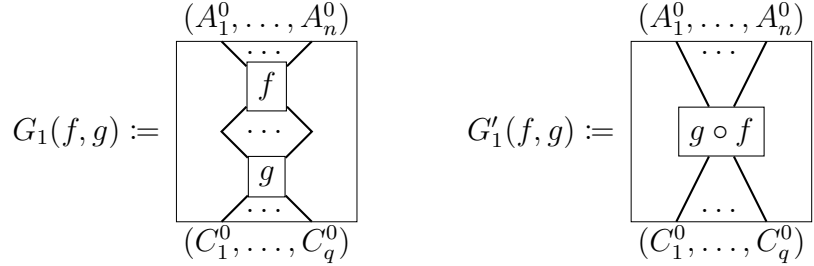
\begin{figure}[H]
   \centering
   \begin{tikzpicture}[every node/.style={node distance=0.3cm},scale=0.6]
      \draw[very thin] (0,0) rectangle (4,4);

      \node[draw,rectangle] at (2,3) (f) {$f$};
      \node[draw,rectangle] at (2,1) (g) {$g$};
      \node at (2,2) {$\dots$};
      \node at (2,0.25) {$\dots$};
      \node at (2,3.75) {$\dots$};

      \draw[thick] (1,4) -- (f) -- (1,2) -- (g) -- (1,0);
      \draw[thick] (3,4) -- (f) -- (3,2) -- (g) -- (3,0);

      \node at (-2,2) {$G_1(f,g) \coloneqq$};

      \node at (2,4.5) {$(A_1^0, \dots, A_n^0)$};
      \node at (2,-0.5) {$(C_1^0, \dots, C_q^0)$};
      
   \end{tikzpicture}
   \hspace{1cm}
   \begin{tikzpicture}[every node/.style={node distance=0.3cm},scale=0.6]
      \draw[very thin] (0,0) rectangle (4,4);

      \node[draw,rectangle] at (2,2) (f) {$g \circ f$};
      \node at (2,0.25) {$\dots$};
      \node at (2,3.75) {$\dots$};

      \draw[thick] (1,4) -- (f) -- (1,0);
      \draw[thick] (3,4) -- (f) -- (3,0);

      \node at (-2,2) {$G_1'(f,g) \coloneqq$};

      \node at (2,4.5) {$(A_1^0, \dots, A_n^0)$};
      \node at (2,-0.5) {$(C_1^0, \dots, C_q^0)$};
   \end{tikzpicture}
\end{figure}
\noindent For all $f \in \cC(A_1 \otimes \dots \otimes A_p, B_1 \otimes
\dots \otimes B_q)$ and $g \in \cC(C_1 \otimes \dots \otimes C_n, D_1
\otimes \dots \otimes D_m)$ we define the following graphs:
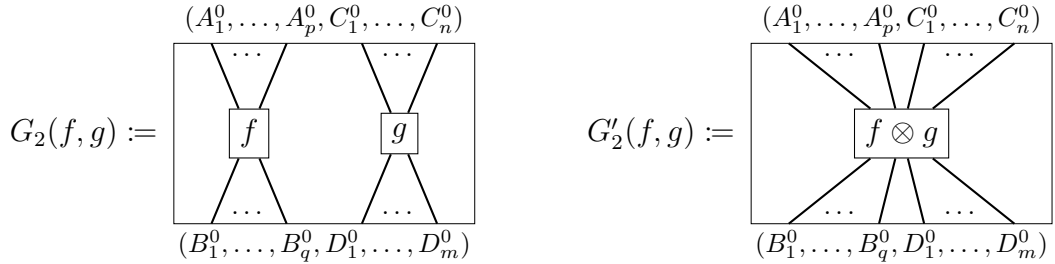
\begin{figure}[H]
   \centering
   \begin{tikzpicture}[every node/.style={node distance=0.3cm},yscale=0.6,xscale=1]
      \draw[very thin] (0,0) rectangle (4,4);

      \node[draw,rectangle] at (1,2) (f) {$f$};
      \node[draw,rectangle] at (3,2) (g) {$g$};
      
      \draw[thick] (0.5,4) -- (f) -- (0.5,0);
      \draw[thick] (1.5,4) -- (f) -- (1.5,0);
      \draw[thick] (2.5,4) -- (g) -- (2.5,0);
      \draw[thick] (3.5,4) -- (g) -- (3.5,0);

     \node at (1,0.25) {\footnotesize $\dots$};
      \node at (1,3.75) {\footnotesize $\dots$};
      \node at (3,0.25) {\footnotesize $\dots$};
      \node at (3,3.75) {\footnotesize $\dots$};

     \node at (2,4.5) {\footnotesize $(A_1^0, \dots, A_p^0, C_1^0, \dots, C_n^0)$};
     \node at (2,-0.5) {\footnotesize $(B_1^0, \dots, B_q^0, D_1^0, \dots, D_m^0)$};

      \node at (-1.2,2) {$G_2(f,g) \coloneqq$};
   \end{tikzpicture}
   \hspace{1cm}
   \begin{tikzpicture}[every node/.style={node distance=0.3cm},yscale=0.6,xscale=1]
      \draw[very thin] (0,0) rectangle (4,4);

      \node[draw,rectangle] at (2,2) (f) {$f \otimes g$};
      
      \draw[thick] (0.5,4) -- (f) -- (0.5,0);
      \draw[thick] (1.7,4) -- (f) -- (1.7,0);
      \draw[thick] (2.3,4) -- (f) -- (2.3,0);
      \draw[thick] (3.5,4) -- (f) -- (3.5,0);
      
      \node at (1.2,0.25) {\footnotesize $\dots$};
      \node at (1.2,3.75) {\footnotesize $\dots$};
      \node at (2.8,0.25) {\footnotesize $\dots$};
      \node at (2.8,3.75) {\footnotesize $\dots$};

     \node at (2,4.5) {\footnotesize $(A_1^0, \dots, A_p^0, C_1^0, \dots, C_n^0)$};
     \node at (2,-0.5) {\footnotesize $(B_1^0, \dots, B_q^0, D_1^0, \dots, D_m^0)$};

            \node at (-1.2,2) {$G_2'(f,g) \coloneqq$};
   \end{tikzpicture}
\end{figure}
\noindent For $A \in \cC$ we define the following graphs:
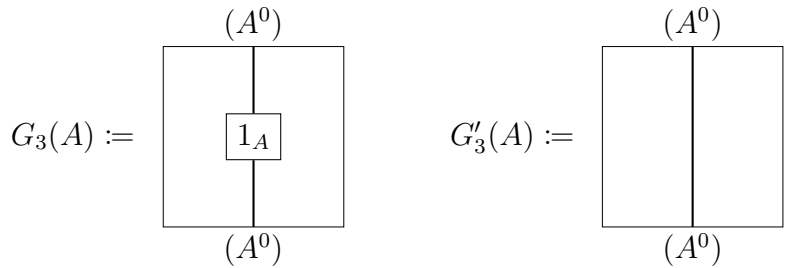
\begin{figure}[H]
   \centering
   \begin{tikzpicture}[every node/.style={node distance=0.3cm},scale=0.6]
      \draw[very thin] (0,0) rectangle (4,4);

      \node[draw,rectangle] at (2,2) (f) {$1_A$};

      \draw[thick] (2,4) -- (f) -- (2,0);

      \node at (-2,2) {$G_3(A) \coloneqq$};
 
      \node at (2,4.5) {$(A^0)$};
      \node at (2,-0.5) {$(A^0)$};
      
   \end{tikzpicture}
   \hspace{1cm}
   \begin{tikzpicture}[every node/.style={node distance=0.3cm},scale=0.6]
      \draw[very thin] (0,0) rectangle (4,4);

      \draw[thick] (2,4)  -- (2,0);

      \node at (-2,2) {$G_3'(A) \coloneqq$};
      \node at (2,4.5) {$(A^0)$};
      \node at (2,-0.5) {$(A^0)$};
   \end{tikzpicture}
\end{figure} The generalized version of this last
replacement pair (with multiple inputs and outputs) will be a
consequence of the three replacement pairs, as it can be obtained as the
$n$-fold product of identities.

\begin{defi}
   Let $A$ and $B$ be two $\cC$-valued yankable graphs. We say that $A$
   reduces to $B$ (denoted by $A \leq B$) when $A$ contains the
   sub-graph $G = G_i(f,g)$ for some $f$ and $g$ arrows of $\cC$ (or $G
   = G_3(X)$ for some $X \in \Ob(\cC)$), and $B$ can be obtained by
   replacing this occurrence of $G$ by $G' = G'_i(f,g)$ (or $G'_3(X)$,
   respectively).
\end{defi}

Another useful relation is $\Gamma \leadsto \Gamma'$, which holds when
there is a deformation of graphs from $\Gamma$ to $\Gamma'$.  Finally we
define the relation $\sim_\cC$ as the reflexive, symmetric and
transitive closure of $\leq \cup \leadsto$.

\begin{defi}
  The category $L(\cC)$ has:
  \begin{itemize}
    \item objects of the form $(A_1^{n_1}, \dots , A_p^{n_p})$ where
      $A_i \in \Ob(\cC)$ and $n_i \in \mathbb{Z}$
    \item morphisms $f : A \to B$, where $f$ is an equivalence class
      under $\sim_\cC$ of $\cC$-valued unit graphs with inputs $A$ and
      outputs $B$.
  \end{itemize}
\end{defi}

\noindent The composition and tensor product in $L(\cC)$ are defined
as in~\cite{joyal1991geometry}.  Let $f \in L(\cC)(A,B)$ and $g \in
L(\cC)(C,D)$ be $\cC$-valued unit graphs.  The morphism $f \otimes g
\in L(\cC)(A \otimes_{L(\cC)} C, B \otimes_{L(\cC)} D)$ is defined by
concatenating horizontally shrunken versions of $f$ and $g$:
$$f \otimes g = \phi_{(0,\frac{1}{2},0,1)}(f) \cup
\phi_{(\frac{1}{2},1,0,1)}(g).$$

Let $f \in L(\cC)(A,B)$ and $g \in L(\cC)(B,C)$ be $\cC$-valued unit
graphs, with $B = (B_1^{n_1}, \dots, B_p^{n_p})$.  We cannot simply
stack the diagrams vertically to define the sequential composite $g
\circ f$, because the horizontal positions $(u_1, \dots, u_p)$ of the
lower outer gates of $f$ might not match with the positions $(v_1,
\dots, v_p)$ of the upper outer gates of $g$. Hence we add identity
links between them:
$$f \circ g = \phi_{(0,1,0,\frac{1}{3})}(f) \cup
\phi_{(0,1,\frac{1}{3},\frac{2}{3})}(L) \cup
\phi_{(0,1,\frac{2}{3},1)}(g)$$ where $L$ is the diagram with identity
links between points $(u_i,1)$ and $(v_i,0)$ where $u_i$ (respectively
$v_i$) is the abscissa of the $i$-th lower outer gate of $g$
(respectively upper outer gate of $f$).
\begin{figure}[H]
  \centering
\begin{tikzpicture}[scale=0.4]
  \draw (0,0) rectangle (4,4);
  \node at (2,2) {$f$};

  \node at (5,2) {$\circ$};
  
  \draw (6,0) rectangle (10,4);
  \node at (8,2) {$g$};

  \node at (11,2) {$=$};

  \draw (12,0) rectangle (16,4);
  \draw (12,1.5) -- (16,1.5);
  \draw (12,2.5) -- (16,2.5);
  \draw (12.7,1.5) -- (13.1,2.5);
  \draw (13,1.5) -- (13.5,2.5);
  \draw (14.7,1.5) -- (14,2.5);
  \draw (15.5,1.5) -- (15.3,2.5);
  \node at (14,0.75) {$f$};
  \node at (14,3.25) {$g$};

  \begin{scope}[xshift=19cm]
  \draw (0,0) rectangle (4,4);
  \node at (2,2) {$f$};

  \node at (5,2) {$\otimes$};
  
  \draw (6,0) rectangle (10,4);
  \node at (8,2) {$g$};

  \node at (11,2) {$=$};

  \draw (12,0) rectangle (16,4);
  \draw (14,0) -- (14,4);
  \node at (13,2) {$f$};
  \node at (15,2) {$g$};

  \end{scope}
\end{tikzpicture}
\end{figure}

\noindent This category is well defined because the composition is
compatible with the relation $\sim_\cC$. In other words, the equivalence
class under $\sim_\cC$ of the vertical stacking of two graphs does not
depend on the choice of the two representatives.  As it is also the case
for the horizontal concatenation, $L(\cC)$ is also strict monoidal.

The product on objects is the concatenation of lists, the unit object
is the empty list denoted by $()$, and the tensor product on arrows is
the horizontal juxtaposition.  Formally, this product is different
from the tensor product of $\cC$. But the rewrite rules defined above
provide a bridge between the two.  If there are objects $A, B, C, D
\in \cC$ such that $A \otimes B = C \otimes D$, then the objects
$(A^0, B^0)$ and $(C^0, D^0)$ are isomorphic in $L(\cC)$, with the
following isomorphism: \begin{figure}[H]
  \centering
   \begin{tikzpicture}[every node/.style={node distance=0.3cm},scale=0.6]
      \node at (-2,2) {$\alpha_{A,B,C,D} =$};

      \draw[very thin] (0,0) rectangle (4,4);

      \node at (1.35,4) (gate1) {};
      \node at (2.65,4) (gate2) {};
      \node[draw,rectangle] at (2,2) (id) {$1_{A\otimes B}$};

      \draw[thick] (gate1.center) -- (id);
      \draw[thick] (gate2.center) -- (id);
      \draw[thick] (id) -- (1.35,0);
      \draw[thick] (id) -- (2.65,0);

      \node at (2,4.5) {$(A^{0}, B^{0})$};
      \node at (2,-0.5) {$(C^{0}, D^{0})$};

      \node at (9,2) {$\alpha_{A,B,C,D}^{-1} = \alpha_{C,D,A,B}$};

      \begin{scope}[xshift=-4cm,yshift=-4cm,yscale=0.5]
        \draw[very thin] (0,-4) rectangle (4,4);

      \node at (1.35,4) (gate1) {};
      \node at (2.65,4) (gate2) {};
      \node[draw,rectangle] at (2,2) (id) {$1_{A\otimes B}$};

      \draw[thick] (gate1.center) -- (id);
      \draw[thick] (gate2.center) -- (id);
      \draw[thick] (id) -- (1.35,0);
      \draw[thick] (id) -- (2.65,0);
      
      \begin{scope}[yshift=-4cm]
      \node at (1.35,4) (gate1) {};
      \node at (2.65,4) (gate2) {};
      \node[draw,rectangle] at (2,2) (id) {$1_{C\otimes D}$};

      \draw[thick] (gate1.center) -- (id);
      \draw[thick] (gate2.center) -- (id);
      \draw[thick] (id) -- (1.35,0);
      \draw[thick] (id) -- (2.65,0);
      \end{scope}

      \node at (2,4.8) {$(A^{0}, B^{0})$};
      \node at (2,-4.9) {$(A^{0}, B^{0})$};

      \node at (4.75,0) {$=$};
      \begin{scope}[xshift=5.5cm]
        \draw[very thin] (0,-4) rectangle (4,4);

      \node at (1.35,4) (gate1) {};
      \node at (2.65,4) (gate2) {};
      \node[draw,rectangle] at (2,0) (id) {$1_{A\otimes B}$};

      \draw[thick] (gate1.center) -- (id);
      \draw[thick] (gate2.center) -- (id);
      \draw[thick] (id) -- (1.35,-4);
      \draw[thick] (id) -- (2.65,-4);

      \node at (2,4.8) {$(A^{0}, B^{0})$};
      \node at (2,-4.9) {$(A^{0}, B^{0})$};

      \node at (4.75,0) {$=$};
     \begin{scope}[xshift=5.5cm]
        \draw[very thin] (0,-4) rectangle (4,4);

      \node at (1.35,4) (gate1) {};
      \node at (2.65,4) (gate2) {};
      \node[draw,rectangle] at (1.35,0) (id1) {$1_A$};
      \node[draw,rectangle] at (2.65,0) (id) {$1_B$};

      \draw[thick] (gate1.center) -- (id1);
      \draw[thick] (gate2.center) -- (id);
      \draw[thick] (id1) -- (1.35,-4);
      \draw[thick] (id) -- (2.65,-4);

      \node at (2,4.8) {$(A^{0}, B^{0})$};
      \node at (2,-4.9) {$(A^{0}, B^{0})$};

      \node at (4.75,0) {$=$};
    \begin{scope}[xshift=5.5cm]
        \draw[very thin] (0,-4) rectangle (4,4);

      \node at (1.35,4) (gate1) {};
      \node at (2.65,4) (gate2) {};

      \draw[thick] (gate1.center) -- (1.35,-4);
      \draw[thick] (gate2.center) -- (2.65,-4);

      \node at (2,4.8) {$(A^{0}, B^{0})$};
      \node at (2,-4.9) {$(A^{0}, B^{0})$};

      \end{scope}

      \end{scope}

      \end{scope}

      \end{scope}

   \end{tikzpicture}
\end{figure}

\begin{prop}
   For any monoidal category $\cC$, $L(\cC)$ is autonomous.
\end{prop}

\begin{proof}
  For any object $(A_1^{n_1}, \dots, A_p^{n_p})$, let us show that it
  has a left adjoint $(A_p^{n_p-1}, \dots, A_1^{n_p-1})$. A similar
  argument shows that it has a right adjoint $(A_p^{n_p+1}, \dots,
  A_1^{n_p+1})$.

  \noindent We define the following morphisms:
  \begin{figure}[H]
   \centering
   \begin{tikzpicture}[every node/.style={node distance=0.3cm,scale=0.96},scale=0.6]
      \draw[very thin] (0,0) rectangle (4,4);

      \tikzstyle{point}=[circle,draw,fill,inner sep=1pt];
      \node at (0.25,4) (gate1) {};
      \node at (1.5,4) (gate2) {};
      \node at (2.5,4) (gate3) {};
      \node at (3.75,4) (gate4) {};
      \node at (2,0.5) (low) {};
      \node at (2,2.5) (high) {};

      \draw[thick] (gate1.center) -- (low.center) -- (gate4.center);
      \draw[thick] (gate2.center) -- (high.center) -- (gate3.center);

      \node at (-1,2) {$\epsilon \coloneqq$};
      \node at (1,5) {$\epsilon : (A_p^{n_p-1}, \dots, A_1^{n_1-1}, A_1^{n_1}, \dots, A_p^{n_p}) \to ()$};

      \node at (1.05,3.7) {$\dots$};
      \node at (2.95,3.7) {$\dots$};
      \node at ($(low) !.7! (high)$) {$\vdots$};

      \begin{scope}[xshift=14cm]
      \draw[very thin] (0,0) rectangle (4,4);

      \tikzstyle{point}=[circle,draw,fill,inner sep=1pt];
      \node at (0.25,0) (gate1) {};
      \node at (1.5,0) (gate2) {};
      \node at (2.5,0) (gate3) {};
      \node at (3.75,0) (gate4) {};
      \node at (2,3.5) (low) {};
      \node at (2,1.5) (high) {};

      \draw[thick] (gate1.center) -- (low.center) -- (gate4.center);
      \draw[thick] (gate2.center) -- (high.center) -- (gate3.center);

      \node at (-1,2) {$\eta \coloneqq$};
      \node at (1,5) {$\eta : () \to (A_1^{n_1}, \dots, A_p^{n_p}, A_p^{n_p-1}, \dots, A_1^{n_1-1})$};

      \node at (1.05,0.3) {$\dots$};
      \node at (2.95,0.3) {$\dots$};
      \node at ($(low) !.5! (high)$) {$\vdots$};
      \end{scope}
      
   \end{tikzpicture}
\end{figure}
  \noindent We emphasize that these graphs are valid $\cC$-valued
  graphs and hence yankable: as they have no inner nodes, the
  condition of Definition~\ref{def:yankable} and the condition (ii) of
  Definition~\ref{def:valued} are vacuously satisfied.  They satisfy
  the yanking equalities, hence the category is
  autonomous.
\end{proof}

\section{Freeness of $L(\cC)$ over $\cC$} \label{sec:freeness}

What remains to do is to show that $L(\cC)$ is the free autonomous
category generated by $\cC$. The first step is to show that $\cC$ can be
embedded functorialy in $L(\cC)$. Then, assuming that $\cC$ is
autonomous, we define the value of a $\cC$-valued graph. Finally we show
that these two constructions are adjoint, hence the freeness of
$L(\cC)$.

\subsection{Functorial embedding of the original category} \label{sec:func-embed}

We define a strongly monoidal functorial embedding of $\cC$ in $L(\cC)$.

\begin{defi}
   Let $F : \cC \to L(\cC)$ be such that for all $A \in \Ob(\cC), F(A) =
   (A^0)$ and \begin{figure}[H]
   \centering
   \begin{tikzpicture}[every node/.style={node distance=0.3cm},scale=0.5]
      \draw[very thin] (0,0) rectangle (4,-4);

      \tikzstyle{point}=[circle,draw,fill,inner sep=1pt];
      \node at (2,-4) (gate1) {};
      \node[draw,rectangle] at (2,-2) (f) {$f$};
      \node at (2,0) (gate3) {};

      \draw[thick] (gate1.center) -- (f) -- (gate3.center);
      \node at (2,0.5) {$(A^0)$};
      \node at (2,-4.5) {$(B^0)$};

      \node at (-1.7,-2) {$F(f) \coloneqq$};

      \begin{scope}[xshift=9cm,yshift=-4cm]
      \draw[very thin] (0,0) rectangle (4,4);

      \node at (1,4) (gate1) {};
      \node at (3,4) (gate2) {};
      \node[draw,rectangle] at (2,2) (id) {$1_{A\otimes B}$};
      \node at (2,0) (gate3) {};
      \node at (1,4.5) {$(A^0,$};
      \node at (3,4.5) {$B^0)$};
      \node at (2,-0.5) {$(A\otimes B^0)$};

      \draw[thick] (gate1.center) -- (id);
      \draw[thick] (gate2.center) -- (id);
      \draw[thick] (id) -- (gate3.center);

      \node at (-2,2) {$\mu_{A,B} \coloneqq$};

     \begin{scope}[xshift=9cm]
      \draw[very thin] (0,0) rectangle (4,4);

      \node[draw,rectangle] at (2,2) (id) {$1_I$};
      \node at (2,0) (gate3) {};

      \node at (2,4.5) {$()$};
      \node at (2,-0.5) {$(I^0)$};

      \draw[thick] (id) -- (gate3.center);

      \node at (-1.5,2) {$\lambda \coloneqq$};

     \end{scope}

      \end{scope}
   \end{tikzpicture}
\end{figure}
\end{defi}

\noindent As a consequence of the first rewrite rule of our relation
$\leq$, $F$ is a functor.  Moreover, it is strongly monoidal, with the
natural isomorphism $\mu_{A,B} : (A^0,B^0) \rightarrow (A \otimes B^0)$
and the isomorphism $\lambda : () \rightarrow (I^0)$.  The coherence
equations translate into the following equalities, which hold from the
rewrite rules defined earlier: \begin{figure}[H]
  \centering
  \begin{tikzpicture}[scale=0.6,every node/.style={scale=0.6}]
    
    \draw (0,0) rectangle (4,4);
    \node[rectangle,draw] at (2,1) (mu1) {$1_{A\otimes B\otimes C}$};
    \node[rectangle,draw] at (1.5,3) (mu2) {$1_{A\otimes B}$};
    \draw (1,4) -- (mu2) -- (mu1) -- (2,0);
    \draw (2,4) -- (mu2);
    \draw (3,4) -- (mu1);

    \node at (4.75,2) {\Large $=$};

    \begin{scope}[xshift=5.5cm]
     \draw (0,0) rectangle (4,4);
     \node[rectangle,draw] at (2,1) (mu1) {$1_{A\otimes B\otimes C}$};
     \node[rectangle,draw] at (2.5,3) (mu2) {$1_{B\otimes C}$};
     \draw (3,4) -- (mu2) -- (mu1) -- (2,0);
     \draw (2,4) -- (mu2);
     \draw (1,4) -- (mu1);
    \end{scope}

    \begin{scope}[xshift=12cm]
    \draw (0,0) rectangle (4,4);
    \node[rectangle,draw] at (2,1) (mu1) {$1_A$};
    \node[rectangle,draw] at (1.5,3) (mu2) {$1_I$};
    \draw (mu2) -- (mu1) -- (2,0);
    \draw (3,4) -- (mu1);

    \node at (4.75,2) {\Large $=$};
    \begin{scope}[xshift=5.5cm]
      \draw (0,0) rectangle (4,4);
      \draw (2,0) -- (2,4);
    \node at (4.75,2) {\Large $=$};
    \end{scope}[xshift=5.5cm]

    \begin{scope}[xshift=11cm]
     \draw (0,0) rectangle (4,4);
     \node[rectangle,draw] at (2,1) (mu1) {$1_{A}$};
     \node[rectangle,draw] at (2.5,3) (mu2) {$1_I$};
     \draw (mu2) -- (mu1) -- (2,0);
     \draw (1,4) -- (mu1);
    \end{scope}
    \end{scope}

  \end{tikzpicture}
\end{figure} 

\subsection{Value of a valued graph} \label{sec:value}

Given a $\cC$-valued graph, we cannot in general interpret this graph as
an arrow of $\cC$, because $\cC$ is not always autonomous. But when
$\cC$ is autonomous, we can use the notion of \textbf{value} $v(\Gamma)$
of a valued graph $\Gamma$. We briefly recall its definition, taken
from~\cite{joyal1988planar} and adapted to our terminology. We start with
the definition of the value of a graph in the monoidal case. To do so,
we restrict our graphs further by requiring that the edges are vertical,
in the following sense.

\begin{defi}
  A \textbf{progressive graph} is a graph $\Gamma$ such that for all
  edge $\gamma$, the projection on the second coordinate of $\gamma$ is
  injective.
\end{defi}

As a consequence, the rotation numbers of edges are null in a progressive
graph.  To define the value of such a graph, we decompose it into
simpler slices.

\begin{defi}
  A \textbf{prime graph} is a progressive graph with exactly one inner
  node, and such that every edge in the graph is connected to it.

  An \textbf{invertible graph} is a progressive graph with no inner
  node.

  An \textbf{elementary slice} is a progressive graph $\Gamma$ that can
  be decomposed as $G_1 \otimes \dots \otimes G_n$ where for each $1
  \leq i \leq n$, $G_i$ is prime or invertible.
\end{defi}
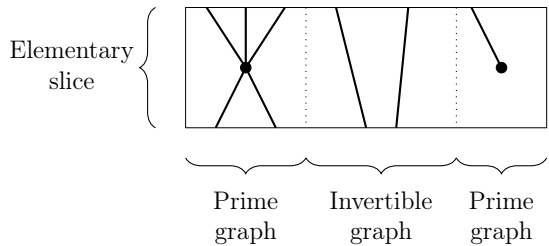
\begin{figure}[H]
  \centering
   \begin{tikzpicture}[every node/.style={node distance=0.3cm,scale=0.7},scale=0.4]
      \draw[very thin] (0,0) rectangle (12,4);

      \node at (0.8,4) (gate1) {};
      \node at (3.2,4) (gate2) {};
      \node[draw,circle,fill,inner sep=2pt] at (2,2) (id) {};

      \draw[thick] (0.7,4) -- (2,2);
      \draw[thick] (2,4) -- (2,2);
      \draw[thick] (3.3,4) -- (2,2);
      \draw[thick] (2,2) -- (1,0);
      \draw[thick] (2,2) -- (3,0);

      \draw[thick] (5,4) -- (6,0);
      \draw[thick] (7.4,4) -- (7,0);
      \node[draw,circle,fill,inner sep=2pt] at (10.5,2) (id2) {};
      \draw[thick] (9.5,4) -- (10.5,2);

      \draw[dotted] (4,4) -- (4,0);

      \draw[dotted] (9,4) -- (9,0);

\begin{scope}[every node/.style={scale=0.8}]
\draw[decorate,decoration={brace,amplitude=6pt}] (-1,0) -- node[left] {\begin{tabular}{c} Elementary \\ slice \end{tabular}} (-1,4);

\draw[decorate,decoration={brace,amplitude=6pt}] (4,-1  ) -- (0,-1  );
\node at (2,-3) {
\begin{tabular}{c} Prime \\ graph \end{tabular}};

\draw[decorate,decoration={brace,amplitude=6pt}] (9,-1) -- (4,-1);
\node at (6.5,-3) {
\begin{tabular}{c} Invertible \\ graph \end{tabular}};

\draw[decorate,decoration={brace,amplitude=6pt}] (12,-1) -- (9,-1);
\node at (10.5,-3) {
\begin{tabular}{c} Prime \\ graph \end{tabular}};
\end{scope}
   \end{tikzpicture}
\end{figure} When these graphs are valued, we can give
them a value $v(\Gamma)$.  The value $v(\Gamma)$ of a prime graph
$\Gamma$ is the value of its unique inner node and the value of an
invertible graph is the identity of its domain (which is equal to its
codomain). Finally, we define the value of an elementary slice $\Gamma =
G_1 \otimes \dots \otimes G_n$ by $v(\Gamma) = v(G_1) \otimes \dots
\otimes v(G_n)$, which is independent of the decomposition.

Notice that any progressive graph $\Gamma$ can be written as $\Gamma =
G_1 \circ \dots \circ G_p$ where $G_i$ are elementary slices. To define
the value of progressive graphs, we need the following lemma, whose
proof is a direct consequence of Proposition~1.1
in~\cite{joyal1991geometry}.
\begin{lemma} \label{lemma:decomposition}
  Let $\Gamma$ be a progressive graph. If $\Gamma = G_1 \circ \dots
  \circ G_p = G_1' \circ \dots \circ G_q'$ where the $G_i$ and $G_j'$
  are elementary slices, then $v(G_1) \circ \dots \circ v(G_p) = v(G_1')
  \circ \dots \circ v(G_q')$
\end{lemma}
\noindent This defines the value of progressive graphs.

The value of general valued graphs is obtained by making the units and
counits explicit in non progressive edges.  More precisely, given a
valued graph $\Gamma$ and $\epsilon > 0$ we define the progressive
graph $\Gamma_\epsilon$ by replacing each non progressive edge as
follows.  First, horizontal segments are eliminated, using the following
replacements (and their upside-down counterparts):
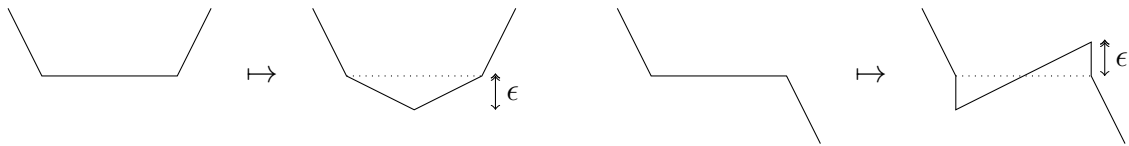
\begin{figure}[H]
  \centering
\begin{tikzpicture}[scale=0.9]

\draw (-1.5,1) -- (-1,0) -- (1,0) -- (1.5,1);
\node at (2.25,0) {$\mapsto$};
\begin{scope}[xshift=4.5cm]
\draw (-1.5,1) -- (-1,0) -- (0,-0.5) -- (1,0) -- (1.5,1);
\draw[dotted] (-1,0) -- (1,0);
\draw[<->] (1.2,0) edge node[right] {$\epsilon$} (1.2,-0.5);
\end{scope}

\begin{scope}[xshift=9cm]
\draw (-1.5,1) -- (-1,0) -- (1,0) -- (1.5,-1);
\node at (2.25,0) {$\mapsto$};
\begin{scope}[xshift=4.5cm]
\draw (-1.5,1) -- (-1,0) -- (-1,-0.5) -- (1,0.5) -- (1,0) -- (1.5,-1);
\draw[dotted] (-1,0) -- (1,0);
\draw[<->] (1.2,0.5) edge node[right] {$\epsilon$} (1.2,0);
\end{scope}
\end{scope}

\end{tikzpicture}
\end{figure} Then, the singular points at turns
are replaced by inner nodes with the appropriate valuation. Let
$\gamma$ be an edge and be $x$ a singular point in $\gamma$ such that
both of its adjacent segments are above $x$. The case where they are
both below is similar. Let $a < b \in [0,1]$ be such that
$[\gamma(a),x]$ and $[x,\gamma(b)]$ are strictly included in the
adjacent segments of $x$.  Finally let $(A,p) = v_1(\gamma)$ and $n =
p + \rho(\gamma_{|[0,a]})$. As the category in which the graph is
valued is autonomous, $A$ has an $(n+1)$-fold right adjoint denoted by
$A^{(n+1)}$, and there is an associated counit $\epsilon_{A^{(n+1)}} :
A^{(n)} \otimes A^{(n+1)} \to I$. We replace $x$ by an inner node
valued by $\epsilon_{A^{(n+1)}}$.  The case for the unit is symmetric.
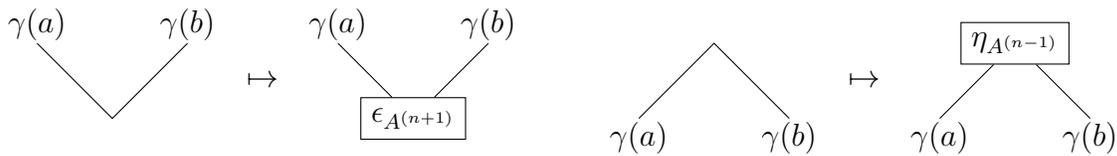
\begin{figure}[H]
  \centering
  \begin{tikzpicture}
    \node at (-1,1.25) {$\gamma(a)$};
    \node at (1,1.25) {$\gamma(b)$};
    \draw (-1,1) -- (0,0) -- (1,1);
    \node at (2,0.5) {$\mapsto$};
    \begin{scope}[xshift=4cm]
   \node at (-1,1.25) {$\gamma(a)$};
    \node at (1,1.25) {$\gamma(b)$};
      \node[rectangle,draw] at (0,0) (n) {$\epsilon_{A^{(n+1)}}$};
      \draw (-1,1) -- (n) -- (1,1);
    \end{scope}
    
    \begin{scope}[xshift=8cm,yshift=1cm,yscale=-1]
    \node at (-1,1.25) {$\gamma(a)$};
    \node at (1,1.25) {$\gamma(b)$};
    \draw (-1,1) -- (0,0) -- (1,1);
    \node at (2,0.5) {$\mapsto$};
    \begin{scope}[xshift=4cm]
   \node at (-1,1.25) {$\gamma(a)$};
    \node at (1,1.25) {$\gamma(b)$};
      \node[rectangle,draw] at (0,0) (n) {$\eta_{A^{(n-1)}}$};
      \draw (-1,1) -- (n) -- (1,1);
    \end{scope}
    \end{scope}
  \end{tikzpicture}
\end{figure}
\noindent Note that these transformations are different from the notion
of replacement introduced earlier in Section~\ref{sec:replacement}.
Here, we only show how to convert a valued graph into a progressive
graph, but do not identify them using an equivalence relation as for the
replacement pairs of Section~\ref{sec:replacement}.

It is shown in~\cite{joyal1988planar} that for $\epsilon$ small enough,
$v(\Gamma_\epsilon)$ is independent of $\epsilon$. We define
$\tilde{v}(\Gamma) \coloneqq v(\Gamma_\epsilon)$ for such an $\epsilon$.
This defines the value of $\cC$-valued graphs when $\cC$ is
autonomous. They also show in their Theorem~4 that this value is
invariant under deformations of graphs.  We state the following lemma
for later use:
\begin{lemma} \label{lemma:progressive}
  Let $f \in \Mor(L(\cC))$ be a $\cC$-valued graph, where $\cC$ is
  autonomous.  There is a $\cC$-valued progressive graph $g$ such that
  $\tilde{v}(f) = \tilde{v}(g)$. Moreover, when $f$ has no horizontal
  segment, $g$ can be obtained from $f$ by replacing singular points by
  inner nodes.
\end{lemma}

In order to make $\tilde{v}$ a strict monoidal functor, we need to prove
the invariance of $\tilde{v}$ under replacement.
\begin{lemma}
  Let $f \in \Mor(L(\cC))$ be a $\cC$-valued graph, where $\cC$ is
  autonomous.  Let $(G^1,G^2)$ be a replacement pair such that
  $\tilde{v}(G^1) = \tilde{v}(G^2)$. Let $t$ be an
  occurrence of $G^1$ in $f$. Then $\tilde{v}(f [ G^1 \coloneqq G^2 ]_t)
  = \tilde{v}(f)$.
\end{lemma}

\begin{proof}
Let $f$ be a valued unit graph, and $t = (a,b,c,d)$ be an occurrence of
$G^1$ in $f$. As $\tilde{v}(f)$ is invariant under deformation of $f$,
we can assume that $f$ contains no horizontal segment.
Up to another regular deformation, $c$ and $d$ are
regular slices. Hence $f$ can be decomposed into graphs $f_1$, $f_2$ and
$f_3$ with boundaries respectively $(0,c)$, $(c,d)$ and $(d,1)$.  We
have $f = f_1 \circ f_2 \circ f_3$ and $f [G^1 \coloneqq G^2]_t = f_1
\circ f_2 [G^1 \coloneqq G^2]_{t'} \circ f_3$, where $t' = (a,b,0,1)$.
Similarly, $f_2$ can be decomposed in the vertical slices $g_1$, $g_2$
and $g_3$, with vertical boundaries respectively $(0,a)$, $(a,b)$,
$(b,1)$.  We have $f_2 [G^1 \coloneqq G^2]_t = g_1 \otimes G^2 \otimes
g_3$ hence $v(f_2 [G^1 \coloneqq G^2]_t) = \tilde{v}(g_1) \otimes
\tilde{v}(G^2) \otimes \tilde{v}(g_3)$ Finally, as $\tilde{v}(G^1) =
\tilde{v}(G^2)$ and $v(f) = \tilde{v}(f_1) \circ \tilde{v}(f_2) \circ
\tilde{v}(f_3)$, we have the required invariance.
\end{proof}
\noindent Hence, as $\tilde{v}$ is compatible with the two relations
$\leq$ and $\leadsto$, it is compatible with $\sim_\cC$. So $\tilde{v}
: L(\cC) \to \cC$ is defined and is a strict monoidal functor.

\subsection{A pair adjoint functors} \label{sec:adjunction}

The objects introduced in our construction can be seen as part of an
adjunction between a free and a forgetful functor. This will show that
$L(\cC)$ has the required categorical properties to be called the free
autonomous category generated by the monoidal category $\cC$.

Let $\Mon$ be the category of strict monoidal categories and strong
monoidal functors between them. We denote by $\Nom$ (as in
auto\emph{nom}ous) the category of autonomous categories and strong
monoidal functors between them.\footnote{Recall that these functors
  automatically preserve adjoints.}  Our construction $L$ corresponds to
a functor from $\Mon$ to $\Nom$.  To make $L$ a functor, we need to
define how it translates a strong monoidal functor $f : \cC \to \cD$ to
$L(f) : L \cC \to L \cD$.  Let $\mu_{(A_1, \dots, A_n)} : f(A_1 \otimes
\dots \otimes A_n) \simeq f(A_1) \otimes \dots \otimes f(A_n)$ be
coherence isomorphism associated to $f$ (with $\mu_{()} : f(I_{\cC})
\simeq I_{\cD}$).  We define $L(f)$ by

\begin{align*}
L(f) :\,& (A_1^{n_1}, \dots, A_p^{n_p}) \mapsto (f(A_1)^{n_1},\dots,
f(A_p)^{n_p})) \\ & (\Gamma,v_0,v_1) \mapsto (\Gamma, v_0', f \circ v_1)
\end{align*}

\noindent We cannot simply define $v_0'$ by $f \circ v_0$ as we have to compose
with the coherence isomorphisms of $f$ to ensure that the domain and
codomain of the valuation match with the product of the valuation of the
incoming and outgoing edges. Let $x$ be a node and $(A_1, \dots, A_p)$
(respectively $(B_1, \dots, B_q)$) be the valuations of its input
(respectively output) edges. We define
$$v_0'(x) = \mu_{(B_1, \dots, B_q)} \circ (f \circ v_0(x)) \circ
\mu^{-1}_{(A_1, \dots, A_p)}$$ One can check that $(\Gamma, v_0', f
\circ v_1)$ is indeed a $\cD$-valued graph and that $L$ is a strict
monoidal functor.

We will show that this functor has a right adjoint $R : \Nom \to \Mon$,
the inclusion functor. We define the unit for this
adjunction by $\eta_{\cC} \coloneqq F_{\cC} : \cC \to R L \cC$ where $F$
is the functorial embedding defined in Section~\ref{sec:func-embed}.

The counit $\epsilon_{\cC} : L R \cC \to \cC$ corresponds to the value
functor introduced in section~\ref{sec:value}. It is defined on objects
by $(A_1^{n_1}, \dots, A_p^{n_p}) \mapsto A_1^{(n_1)} \otimes \dots
\otimes A_p^{(n_p)}$. In other words, the formal product is sent to the
actual product of the original category, and the formal adjoints are
sent to the actual adjoints.  An arrow $f$, that is to say a
$\cC$-valued graph, is sent to its value $v(f)$.

We now move on to the proof of the unit-counit equations, starting with
$(R \epsilon) \circ (\eta R) = 1_R$.  First, $\eta_{R(\cD)}$ takes an
arrow in an autonomous category, seen as a monoidal category, and
bundles it in a diagram. Then, $R \epsilon_\cD$ evaluates this diagram
in $\cD$ (which is possible because $\cD$ is actually autonomous), and
the result is seen as an arrow in $R \cD$.  \begin{figure}[H]
  \centering
   \begin{tikzpicture}[every node/.style={node distance=0.3cm},scale=0.5]
      \draw[very thin] (0,0) rectangle (4,4);

      \tikzstyle{point}=[circle,draw,fill,inner sep=1pt];
      \node at (2,4) (gate1) {};
      \node[draw,rectangle] at (2,2) (id) {$f$};
      \node at (2,-0.5) {$(B^0)$};
      \node at (2,4.5) {$(A^0)$};

      \draw[thick] (gate1.center) -- (id);
      \draw[thick] (id) -- (2,0);

      \node at (-6,2) (f1) {$f \in \cC(A,B)$};
      \draw[|->] (-3,2) -- (-1,2);
      \node at (-2,2.5) {$\eta_{R(\cD)}$};

      \node at (8,2) (f) {$f$};
      \draw[|->] (5,2) -- (7,2);
      \node at (6,2.5) {$R(\epsilon_\cD)$};
      
   \end{tikzpicture}
\end{figure}
\noindent Hence the composition of the two is the identity.  Let us show
the remaining equality: $(\epsilon L) \circ (L \eta) = 1_L$.

\begin{align*}
  (L \eta)_\cC = L (\eta_\cC) &: L \cC \to L R L\cC \\ & (A_1^{n_1},
  \dots, A_p^{n_p}) \mapsto ((A_1^0)^{n_1}, \dots, (A_p^0)^{n_p})
\end{align*}
The functor $L(\eta_\cC)$ applies $\eta_\cC$ to the valuations of the
graph, and the result is composed with the coherence morphisms so that
the domains and codomains match with the incoming and outgoing
edges. Graphically, this consists in adding inner boxes in each node,
with morally the same inputs and outputs as the outer box it is
contained in.  Then, $\epsilon_{L(\cC)}$ evaluates the resulting graph
in $L(\cC)$.

\begin{thm}
  The following equality holds: $(\epsilon L) \circ (L \eta) = 1_L$.
\end{thm}

\begin{proof}
Let $f$ be an arrow in $L \cC$. Up to a deformation described in
Section~\ref{sec:value}, we can assume that it has no horizontal
segment. As $RL \cC$ is autonomous, we can apply
Lemma~\ref{lemma:progressive} to $L \eta_\cC (f) \in \Mor(LRL \cC)$:
$\tilde{v}(L \eta_\cC (f)) = \tilde{v}(g)$ where $g$ is progressive and
is obtained from $L \eta_\cC (f)$ by replacing singular points by inner
nodes.  But $L \eta_\cC (f)$ differs only from $f$ by the valuations:
the underlying graph is the same. Hence the decomposition of $g$ into
prime and invertible factors given by Lemma~\ref{lemma:decomposition}
induces a decomposition of $f$, where the factors are not necessarily
prime or invertible however.

We prove that $\tilde{v}(g) = f$ by induction on the number of factors
in the decomposition of $g$.  If $g$ is prime, let $x$ be its unique
inner node. There are two cases.  If $x$ is also an inner node in $f$,
then $g$ is indeed mapped to itself, as shown in Figure~\ref{fig:l-eta}.
\begin{figure}[p]
  \centering
   \begin{tikzpicture}[every node/.style={node distance=0.3cm},scale=0.6]
   \tikzstyle{point}=[circle,draw,fill,inner sep=1pt];

   \foreach \xs/\ys in {0cm/0cm,22cm/0cm} {
      \begin{scope}[xshift=\xs]
      \draw[very thin] (0,0) rectangle (4,4);

      \node at (0.8,4) (gate1) {};
      \node at (3.2,4) (gate2) {};
      \node[draw,rectangle] at (2,2) (id) {$f$};

      \draw[thick] (gate1.center) -- (id);
      \draw[thick] (gate2.center) -- (id);
      \draw[thick] (id) -- (0.8,0);
      \draw[thick] (id) -- (3.2,0);

      \node at (2,3.5) {$\dots$};
      \node at (2,0.5) {$\dots$};
      \node at (2,4.5) {$(A_1^{n_1}, \dots, A_p^{n_p})$};
      \node at (2,-0.5) {$(B_1^{n_1}, \dots, B_q^{m_q})$};
      \end{scope}
      }    
  
      \foreach \x/\lbl in {0cm/$L(\eta_\cC)$,13cm/$\epsilon_{L(\cC)}$} {
        \begin{scope}[xshift=\x]
      \draw[|->] (5.5,2) -- (7.5,2);
      \node at (6.5,2.5) {\lbl};
      \end{scope}
      }

      \begin{scope}[xshift=9cm,yshift=-2cm,scale=2]
      \draw[very thin] (0,0) rectangle (4,4);

      \tikzstyle{point}=[circle,draw,fill,inner sep=1pt];
      \node at (0.8,4) (gate1) {};
      \node at (3.2,4) (gate2) {};

      \begin{scope}[scale=0.5,xshift=2cm,yshift=2cm]
        \node[rectangle,draw,very thin,minimum width=1.8cm,minimum height=1.8cm,inner sep=0pt] at (2,2) (id) {};

        \tikzstyle{point}=[circle,draw,fill,inner sep=1pt];
        \node[draw,rectangle] at (2,2) (id2) {$f$};

        \draw[thick] ($(id.north)+(-0.8,0)$) -- (id2);
        \draw[thick] ($(id.north)+(0.8,0)$) -- (id2);
        \draw[thick] (id2) -- ($(id.south)+(-0.8,0)$);
        \draw[thick] (id2) -- ($(id.south)+(0.8,0)$);

       \node at (2,4) {\footnotesize $(A_1^{n_1}, \dots, A_p^{n_p})$};
       \node at (2,0) {\footnotesize $(B_1^{m_1}, \dots, B_q^{m_q})$};
      \end{scope}
      \node[rectangle,draw,very thin,minimum width=3cm,minimum height=3cm] at (2,2) (id) {};

      \draw[thick] (gate1.center) node[above] (a1) {$((A_1^0)^{n_1},$} -- (id);
      \draw[thick] (gate2.center) node[above] (ap) {$, (A_p^0)^{n_p})$} -- (id);
      \draw[thick] (id) -- (0.8,0) node[below] (b1) {$((B_1^0)^{m_1},$};
      \draw[thick] (id) -- (3.2,0) node[below] (bq) {$, (B_q^0)^{m_q})$};

      \node at ($(a1) !.5! (ap)$) {$\dots$};
      \node at ($(b1) !.5! (bq)$) {$\dots$};

      \node at (2,3.7) {$\dots$};
      \node at (2,0.3) {$\dots$};
      \end{scope}
      
   \end{tikzpicture}
   \caption{Equality in the case of a prime diagram  \label{fig:l-eta}}
\end{figure}
Otherwise, $x$ corresponds to a singular point in $f$ and is
labelled by a unit or a counit, and is mapped to itself as shown
in Figure~\ref{fig:l-eta-cup}.
\begin{figure}[p]
  \centering
   \begin{tikzpicture}[every node/.style={node distance=0.3cm},scale=0.4]
   \tikzstyle{point}=[circle,draw,fill,inner sep=1pt];

   \node at (-2.3,2) {$f =$};
   \foreach \xs/\ys in {0cm/0cm,23cm/-4.5cm} {
      \begin{scope}[xshift=\xs,yshift=\ys]
      \draw[very thin] (0,0) rectangle (4,4);

      \node at (0.8,4) (gate1) {};
      \node at (3.2,4) (gate2) {};

      \draw[thick] (gate1.center) -- (2,2);
      \draw[thick] (gate2.center) -- (2,2);

      \node at (0.8,4.5) {$(A^{n},$};
      \node at (3.2,4.5) {$A^{n+1})$};
      \node at (2,-0.5) {$()$};
      \end{scope}
      }

      \begin{scope}[xshift=11cm,yshift=0cm]
      \draw[very thin] (0,0) rectangle (4,4);

      \node at (0.8,4) (gate1) {};
      \node at (3.2,4) (gate2) {};

      \draw[thick] (gate1.center) -- (2,2);
      \draw[thick] (gate2.center) -- (2,2);

      \node at (2,4.7) {$((A^0)^n, (A^0)^{n+1})$};
      \node at (2,-0.5) {$()$};
      \end{scope}
  
        \begin{scope}[xshift=1cm]
      \draw[|->] (5.2,2) -- (7.8,2);
      \node at (6.5,2.7) {$L \eta_\cC$};
      \end{scope}

       \begin{scope}[xshift=13cm]
      \draw[|->] (5.2,1) -- (8,-1);
      \node at (7.2,0.6) {$\epsilon_{L(\cC)}$};
      \end{scope}

       \begin{scope}[xshift=13cm,yshift=-5cm,yscale=-1]
      \draw[|->] (5.2,1) -- (8,-1);
      \node at (7.2,0.6) {$\epsilon_{L(\cC)}$};
      \end{scope}
      \node at (6.7,-7) {$g = $};
      \begin{scope}[xshift=9cm,yshift=-11cm,scale=2]
      \draw[very thin] (0,0) rectangle (4,4);

      \tikzstyle{point}=[circle,draw,fill,inner sep=1pt];
      \node at (0.8,4) (gate1) {};
      \node at (3.2,4) (gate2) {};

      \begin{scope}[scale=0.5,xshift=2cm,yshift=2cm]
        \node[rectangle,draw,very thin,minimum width=1.2cm,minimum height=1.2cm,inner sep=0pt] at (2,2) (id) {};

        \tikzstyle{point}=[circle,draw,fill,inner sep=1pt];

        \draw[thick] ($(id.north)+(-0.8,0)$) -- (2,2);
        \draw[thick] ($(id.north)+(0.8,0)$) -- (2,2);

       \node at (2,4) {\footnotesize $(A^{n}, A^{n+1})$};
       \node at (2,0) {\footnotesize $()$};

      \end{scope}
      \node[rectangle,draw,very thin,minimum width=2cm,minimum height=2cm] at (2,2) (id) {};

      \draw[thick] (gate1.center) node[above] (a1) {$((A^0)^{n},$} -- (id);
      \draw[thick] (gate2.center) node[above] (ap) {$ (A^0)^{n+1})$} -- (id);
      \node at (2,-0.25) {$()$};

      \end{scope}
      
   \end{tikzpicture}
   \caption{Equality in the case of a counit \label{fig:l-eta-cup}}
\end{figure}
If $g$ is invertible, it is mapped to itself as well.

Now for the general case, suppose that $g = g_1 \circ g_2$ where $g_1$
and $g_2$ can be decomposed in a smaller number of factors (the case
$g_1 \otimes g_2$ is analogous).  As noted earlier, this induces a
decomposition $L \eta_\cC (f) = h_1 \circ h_2$ such that $g_i$ is the
progressive version of $h_i$.  This induces in turn a decomposition $f
= f_1 \circ f_2$ such that $h_i = L \eta_\cC (f_i)$. By induction,
$\tilde{v}(g_i) = f_i$. As $\tilde{v}$ is a strict monoidal functor, we get
$\tilde{v}(g) = f$.
\end{proof}

The notion of adjunction helps us to relate our construction with that
of~\cite{preller2007free} who describe the free
autonomous category generated by a category. Let $L' : \Cat \to \Mon$ be
the free monoidal category functor, and $R' : \Mon \to \Cat$ be the
corresponding forgetful functor.  By composition of the adjunctions, $L
\circ L'$ is left adjoint to $R' \circ R$.

\begin{figure}[H]
  \centering
  \begin{tikzpicture}
    
    \node (mon) {$\Mon$};
    \node[right of=mon, node distance=3cm] (nom) {$\Nom$};
    
    \draw[-latex] (mon) edge[bend left] node[above] {$L$} (nom);
    \draw[-latex] (nom) edge[bend left] node[below] {$R$} (mon);
    \node at ($(mon) !.5! (nom)$) {$\perp$};

    \node[left of=mon, node distance=3cm] (cat) {$\Cat$};
    \draw[-latex] (cat) edge[bend left] node[above] {$L'$} (mon);
    \draw[-latex] (mon) edge[bend left] node[below] {$R'$} (cat);
    \node at ($(cat) !.5! (mon)$) {$\perp$};

  \end{tikzpicture}
\end{figure}

The construction of Preller and Lambek corresponds to a free functor from
$\Cat$ to $\Nom$, which is equivalent to $L \circ L'$ by uniqueness of
the adjoint. Hence we just gave a factorization of their free functor.

\end{document}